\documentclass[11pt,letterpaper]{article}
\usepackage{fullpage}

\usepackage{nicefrac}
\usepackage{color}

\usepackage{amsmath,amssymb,amsthm}
\usepackage{graphicx,amsfonts}
\usepackage[caption=false]{subfig}
\usepackage{tikz,pgfplots}
\usetikzlibrary{shapes.arrows,chains,positioning}

\newcommand{\Matlab}{\textsc{MATLAB}}
\newcommand{\prox}{\mathbf{prox}}
\newcommand{\diag}{\mathbf{diag}}
\DeclareMathOperator{\proj}{\mathcal{P}}
\DeclareMathOperator{\dom}{\mathbf{dom}}
\DeclareMathOperator*{\argmin}{arg\,min}

\DeclareMathOperator{\sgn}{sgn}

\newcommand{\reals}{\mathbb{R}}

\newcommand{\tnabla}{\widetilde{\nabla}}
\newcommand{\dist}{\mathrm{dist}}
\newcommand{\ie}{i.e.}
\newcommand{\eg}{e.g.}
\def\PC{\proj_{\mathcal C}}

\theoremstyle{plain}
\newtheorem{proposition}{Proposition}
\newtheorem{assumption}{Assumption}
\theoremstyle{remark}
\newtheorem{remark}{Remark}

\begin{document}

\title{Generalized Row-Action Methods for Tomographic Imaging\thanks{This work is part of the project
High-Definition Tomography and it is supported
by Grant No.~ERC-2011-ADG\_20110209 from the European Research Council.}
}

\author{Martin S. Andersen\thanks{
              Department of Applied Mathematics and Computer Science, 
              Technical University of Denmark,
              \texttt{\{mskan,pcha\}@dtu.dk}.} \and Per Christian Hansen\footnotemark[2]}

\maketitle

\begin{abstract} Row-action methods play an important role in
  tomographic image reconstruction. Many such methods can be viewed as
  incremental gradient methods for minimizing a sum of a large number
  of convex functions, and despite their relatively poor global rate
  of convergence, these methods often exhibit fast initial convergence
  which is desirable in applications where a low-accuracy solution is
  acceptable. In this paper, we propose relaxed variants of a class of
  incremental proximal gradient methods, and these variants generalize
  many existing row-action methods for tomographic imaging. Moreover,
  they allow us to derive new incremental algorithms for tomographic
  imaging that incorporate different types of prior information via
  regularization. We demonstrate the efficacy of the approach with
  some numerical examples.

\bigskip
\noindent
\textbf{Keywords}\quad Incremental methods, inverse problems,
regularization, tomographic imaging
\end{abstract}

\section{Introduction}
\label{intro}

Tomographic imaging is an indispensable non-invasive measurement technique
for diagnostics, exploration, analysis, and design;
see \cite{BeBo98}, \cite{Herm09}, \cite{Natt01} and the references therein.
Discretizations of tomographic imaging problems often lead to large sparse systems
of linear equations with noisy data:
  \begin{equation}
  \label{eq:Axb}
    A\, x \simeq b, \qquad  A\in\mathbb{R}^{m \times n}.
  \end{equation}
Here the vector $x$ represents the unknown image,
the vector $b$ is the given (usually inaccurate/noisy) data, and
the matrix $A$ models the forward problem.
There are no restrictions on the dimensions of $A$, and both over- and underdetermined systems
arise in applications, depending on the amount of data generated in a given experiment.

Iterative algorithms are often well-suited for solving the large-scale problem~(\ref{eq:Axb}),
and several classes of methods have emerged \cite{ElNP10,HeMe93}.
They all produce regularized solutions that approximate the exact and unknown solution image
without being too sensitive to the perturbation of the data.

This work focuses on a specific class of so-called row-action methods,
the basic form of which is known as Kaczmarz's method or ART
(algebraic reconstruction technique) \cite{GBH:70,Kac:37}. These
methods have been used for several decades as the core computational
routines for tomographic imaging, and they are recognized for often
having fast initial convergence towards the desired image.  An
important advantage of these methods is that they access the matrix
$A$ one row---or one block---at a time, thus making the methods well
suited for modern computer architectures.

Several extensions of the classical (block) ART methods have been
proposed with the goal of improving certain characteristics of the
reconstructed images.  Of particular interest is the use of total
variation (TV) regularization as a way to better preserve edges and
detail in the image.  For example, Censor, Davidi, and Herman et al.\
\cite{DHC:09} developed a so-called ``perturbation resilient''
framework to incorporate TV regularization into the ART iterations,
while Sidky and Pan \cite{SiP:08} proposed a hybrid algorithm where
ART is combined with the steepest descent method, also to incorporate
TV regularization.

The main goal of this paper is to provide a theoretical and
algorithmic framework for studying and generalizing the ART methods.
The cornerstone of our approach is an interpretation of ART as a
so-called incremental proximal gradient method for convex
optimization.  This allows us to generalize the method (e.g. with the TV
regularization term) in a rigorous way---thus avoiding the heuristic
arguments sometimes found in applications.

The main contribution of this paper is twofold: (i) we propose a
generalization of the incremental proximal gradient framework of
Bertsekas \cite{Ber:11b,Ber:11} that includes a relaxation parameter, and (ii)
using this framework, we propose a class of generalized row-action
methods that allows us to incorporate different kinds of prior
information in the reconstruction problem via regularization.

The paper is organized as follows. In Section~\ref{sec:optimization},
we discuss incremental methods and proximal methods for convex
optimization, and in Section~\ref{sec:relaxed}, we present two relaxed
incremental proximal gradient methods. We discuss some connections
between existing row-action methods and the relaxed incremental
proximal gradient framework in Section~\ref{sec:art}, and in
Section~\ref{sec:artreg}, we consider generalized row-action methods
for data fitting with a regularization term. We present some numerical
results in Section~\ref{sec-numres}, and we conclude the paper in
Section~\ref{sec-conclusions}.

\textit{Notation.}
The $i$th row of $A$ is denoted by $a_i^T$, and
$A^\dagger$ denotes the Moore--Penrose pseudoinverse of $A$.
Given a convex function $f: \reals^n \to \reals \cup \{ \pm \infty \}$,
we denote by $\dom f = \{x \in \reals^n \,|\, f(x) < \infty \}$ the effective
domain of $f$. Finally, $\PC(x)$ denotes the Euclidean projection of
$x\in \reals^n$ on a closed convex set $\mathcal C$ of $\reals^n$, and
$\dist(x,\mathcal C)\equiv \|x - \PC(x) \|_2$ denotes is the Euclidean
distance from $x$ to the set $\mathcal C$.

\section{The Optimization Framework}
\label{sec:optimization}

Many reconstruction problems in tomographic imaging can be expressed as a constrained convex
optimization problem with an objective function that is given as a sum of $m$ convex functions, \ie,
\begin{align}\label{e-incremental-form}
  \begin{array}{ll}
    \mbox{minimize}
    & f(x) \equiv \sum_{i=1}^m f_i(x) \\[1mm]
    \mbox{subject to}
    & x \in \mathcal{C}.
  \end{array}
\end{align}
Here $x \in \reals^n$ is the optimization variable, $f_i : \reals^n
\to \reals$, and $\mathcal C$ is a closed convex subset of $\reals^n$.
The functions $f_i$ 
typically represent 
data-fidelity terms, such as squared residuals and one or more
regularization terms that incorporate prior knowledge.  The set
$\mathcal C$ may represent bounds on the components of $x$; in this
paper we will simply assume that the projection operator $\PC(\cdot)$
associated with the set $\mathcal C$ is cheap to evaluate. 

If the objective function in \eqref{e-incremental-form} is
differentiable with a Lipschitz continuous gradient, the problem can
be solved using an accelerated gradient projection method.
If $x_k$ denotes the $k$th iterate and $f^\star$ is the minimum, then
the error $f(x_k)-f^\star$ is $O(1/k^2)$; see, \eg, \cite{Nes:04}.
For problems with a nonsmooth objective function that is Lipschitz
continuous on a bounded set, the error bound is
$O(1/\sqrt{k})$, and this can be achieved using a projected
subgradient method with a diminishing step-size rule. This
error bound can often be improved by exploiting problem structure.
For example, the accelerated proximal gradient method of Beck and Teboulle
\cite{BeT:09} splits the objective function into a smooth term and a
nonsmooth term, and this method achieves the same error bound as the
accelerated methods for smooth optimization, namely $O(1/k^2)$.

Before we turn to the main subject of the paper in \S\ref{sec:relaxed}, relaxed
incremental proximal gradient methods, we briefly review some necessary material.

\subsection{Incremental Gradient Methods}
\label{subsec:incremental}

When the objective function in \eqref{e-incremental-form} is comprised
of a very large number of functions, 
the cost of computing the gradient (or a subgradient) may be very high. To avoid
computing the full gradient, incremental gradient methods use only the
gradient of a single component of the objective function at iteration
$k$, \ie,
  \[
    x_{k+1} = \PC \left( x_k - t_k \nabla f_{i_k}(x_k) \right) ,
  \]
where $i_k \in \{1,\ldots,m\}$ is the index of the
component used for the update at iteration $k$.
The index $i_k$ is commonly chosen either in a cyclic manner
(\eg, $i_k = (k \mod m) +1$) 
or drawn uniformly at random; 
another possibility is to combine the
cyclic rule with randomization by shuffling the order of the indices
at the beginning of each cycle, 
and empirical
evidence suggests that this works very well in practice~\cite{Ber:11,ReR:12}.

Incremental methods typically have a very slow \emph{asymptotic} rate of
convergence, and like subgradient methods they require a diminishing
step-size rule to ensure convergence.
In tomographic applications, however, we
are more interested in the \emph{initial} rate of convergence
(and the associated semi-convergence \cite{Natt01}), and
this can be very fast for incremental methods compared to their
nonincremental counterparts \cite{Ber:97,Ber:99,Ber:11,EHN:13}.
There are also several examples of hybrid methods
that gradually transition from an incremental method
to a full gradient method in order to combine the fast initial
convergence of the incremental method and the asymptotic rate of the
full gradient method; see, \eg, \cite{Ber:97,BHG:07,FrS:12} and references therein.

\subsection{Proximal Methods}
\label{subsec:proximal}

Given a closed convex function $f : \reals^n \to \reals$,
the proximal operator $\prox_{f}(x) : \reals^n \to \reals^n$ associated with $f$
is defined as follows \cite{Mor:65}
  \begin{align}
  \label{e-prox-map}
    \prox_{f}(x) = \argmin_{u \in \reals^n} \left\{ f(u) + \nicefrac{1}{2} \| u - x \|_2^2 \right\}.
  \end{align}
The first-order optimality condition associated with the minimization
in \eqref{e-prox-map} can be expressed as
$x - u \in \partial f(u)$,
where $\partial f(u)$ denotes the subdifferential of $f$ at $u \in \dom f$,
defined as
  \begin{align}\label{e-subdifferential}
    \partial f(u) = \{w \in \reals^n \,|\, f(y) \geq f(u) + w^T(y-u) \}.
  \end{align}
In particular, if $f$ is differentiable at $u$, then $\partial f(u)$
is the singleton $\{\nabla f(u)\}$ where $\nabla f(u)$ denotes the
gradient of $f$ at $u$.  It follows 
that if $x$ is a fixed-point of $\prox_f(x)$ (\ie, if $x = \prox_f(x)$),
then $0 \in \partial f(x)$ and hence $x$ is a minimizer of $f$. In
other words, minimizing $f$ is equivalent to finding a fixed-point of
$\prox_f(x)$.

The proximal operator associated with the indicator function of a
closed convex set $\mathcal C$ of $\reals^n$, defined as
  \[
    I_{\mathcal C}(x) =
    \begin{cases}
      0 & x \in \mathcal C\\
      \infty & \text{otherwise,}
    \end{cases}
  \]
is simply the Euclidean projection of $x$ on $\mathcal C$, \ie,
$\prox_{I_\mathcal C}(x) = \argmin_{x \in \mathcal C} \| u - x\|_2^2 = \PC(x)$.
It is therefore natural to view the proximal operator associated with a closed
convex function $f$ as a generalized projection operator.

The proximal point method, proposed by Martinet
\cite{Mar:70,Mar:72} in the early 1970s and further studied by
Rockafellar \cite{Roc:76}, is a method for solving monotone inclusion
problems of the form $0 \in T(x)$ where $T$ is a maximal monotone operator.
Since the subdifferential operator $\partial f$ associated
with a closed convex function $f$ is maximal monotone~\cite{Roc:70a},
the proximal point algorithm can be used to solve the
inclusion problem $0 \in \partial f(x)$. For this problem, the
proximal point algorithm can be written as 
  \begin{align}\label{e-prox-point}
    x_{k+1} = \prox_{t_k f}(x_k) = \argmin_{u \in \reals^n} \left\{ t_k
    f(u) +\nicefrac{1}{2} \|u - x_k\|_2^2 \right\} ,
  \end{align}
where $\{t_k\}$ is a sequence of positive parameters. It follows from
the optimality condition $(x_k-u)/t_k \in \partial f(u)$ associated with
the minimization in \eqref{e-prox-point} that the proximal point
algorithm can be expressed as
  \[
    x_{k+1} = x_k - t_k \tnabla f(x_{k+1}) ,
  \]
where $\tnabla f(x_{k+1}) = (x_k-x_{k+1})/t_k$ is a subgradient that belongs to the
subdifferential $\partial f (x_{k+1})$, and $t_k$ is an \emph{implicit}
step-size parameter. Thus, the proximal point method
can be viewed as an \emph{implicit} (sub)gradient method, and unlike
the standard gradient method, it converges for any positive sequence $\{t_k\}$,
provided that a minimum exists.

\subsection{Incremental Proximal Gradient Methods}
\label{subsec:relaxed}

The incremental proximal gradient methods of Bertsekas \cite{Ber:11}
seek to minimize a sum $f(x) = \sum_{i=1}^m f_{i}(x)$ over a closed
convex subset $\mathcal C$ of $\reals^n$, and each $f_i$ is a sum of two convex functions $f_i(x) = g_i(x) + h_i(x)$
with $g_i : \mathcal \reals^n \to \reals$ and $h_i : \reals^n \to
\reals$, \ie,
  \begin{align} \label{e-opt-sum-of-functions}
    \begin{array}{ll}
      \mbox{minimize}
      & f(x) \equiv \sum_{i=1}^m \bigl( g_i(x) + h_i(x) \bigr) \\[1mm]
      \mbox{subject to}
      & x \in \mathcal C.
    \end{array}
  \end{align}
We will assume that each of the functions $g_i$ possesses ``favorable structure'' so
that the proximal minimization $\prox_{g_i}(x)$ is easy
to solve or has a closed-form solution.

Given the $k$th iterate $x_{k}$, an index $i_k \in \{1,\ldots,m\}$,
and a step size $t_k > 0$, Bertsekas' first incremental proximal
method, which we denote IPG1, computes $x_{k+1}$ as follows
\begin{subequations}
  \begin{align}
    \label{e-ipg1-1}
\text{(IPG1)}& &    z_{k} &= \prox_{t_k g_{i_k}}( x_{k} ) && \\
    \label{e-ipg1-2}
&&    x_{k+1} &= \PC \bigl( z_{k} - t_k \tnabla
      h_{i_k}(z_{k}) \bigr). &&
  \end{align}
\end{subequations}
Here $\tnabla h_{i_k}(z_{k})$ denotes a subgradient of $h_{i_k}$ at
$z_{k}$, \ie, $\tnabla h_{i_k}(z_{k}) \in \partial
h_{i_k}(z_{k})$. From the optimality conditions associated with \eqref{e-ipg1-1} and
since $z_{k}$ is unique, we have that $z_{k} = x_{k} - t_k
\tnabla g_{i_k}(z_{k})$ for some $\tnabla g_{i_k}(z_{k})$ in the
subdifferential of $g_{i_k}$ at $z_{k}$. Substituting this expression for $z_{k}$ in
\eqref{e-ipg1-2}, we obtain the following equivalent expression for $x_{k+1}$:
\begin{align}
  \label{e-ipg1-2-alt}
  x_{k+1} &= \PC \bigl( x_{k} - t_k \tnabla  f_{i_k}(z_{k}) \bigr).
\end{align}
Thus, IPG1 can be viewed as an incremental extragradient-like method
where $x_{k+1}$ is obtained by first computing a ``predictor''
$z_{k}$, followed by a subgradient step based on a subgradient of
$f_{i_k}$, evaluated at the predictor $z_{k}$ instead of at the
current iterate $x_{k}$.

The second incremental proximal method of Bertsekas, which we will
refer to as IPG2, can be expressed as the iteration
\begin{subequations}
  \begin{align}
    \label{e-ipg2-1}
\text{(IPG2)} &&
    z_{k} &= x_{k} - t_k \tnabla h_{i_k}(x_{k}) &&\\
    \label{e-ipg2-2}
 &&   x_{k+1} &= \prox_{t_k \tilde g_{i_k}} (z_{k} ), &&
  \end{align}
\end{subequations}
where $\tilde g_{i_k}(x) = g_{i_k}(x) + I_{\mathcal C}(x)$. If we substitute \eqref{e-ipg2-1} for $z_{k}$ in \eqref{e-ipg2-2}, we obtain
the equivalent formulation
\begin{align}
  \label{e-ipg2-alt}
  x_{k+1} &= \argmin_{x \in \mathcal C} \Bigl\{ g_{i_k}(x) +
    \tnabla h_{i_k}(x_{k})^Tx +
      \frac{1}{2t_k} \|x - x_{k}\|_2^2 \Bigr\}.
\end{align}
As pointed out in \cite{Ber:11}, IPG2 can be viewed as an incremental version
of the iterative shrinkage/thresholding algorithm \cite{CDNL:98,DDD:04}.
Alternatively, IPG2 can be viewed as an incremental
proximal algorithm with \emph{partial} linearization of the component
function $f_i$ (\ie, only $h_i$ is linearized).

Both of the methods IPG1 and IPG2 reduce to the same incremental (sub)gra\-dient
algorithm if $g_{i}(x) = 0$ for all $i$.
However, if $h_i(x) = 0 $ for
all $i$, we obtain two slightly different incremental proximal methods
in general: IPG1 involves an unconstrained optimization which is followed by
an explicit projection on $\mathcal C$ whereas IPG2 includes the
constraint $x \in \mathcal C$ in the minimization \eqref{e-ipg2-alt}.

\section{Relaxed Incremental Proximal Gradient Methods}
\label{sec:relaxed}
It is well-known that the performance of many classical row-action
methods for tomographic imaging depend strongly on a relaxation
parameter. Motivated by this, we now propose relaxed variants of the
incremental proximal gradient methods.

\subsection{A Modified IPG2 Method}

In some applications, the constraint $x\in
\mathcal C$ in \eqref{e-ipg2-alt} prohibits a closed-form solution or
cheap computation of the solution to the proximal minimization, and in
such applications, IPG1 may be more suitable than IPG2. To overcome
this limitation of IPG2, we propose a modified variant of IPG2 which
omits the constraint $x \in \mathcal C$ from the minimization
\eqref{e-ipg2-alt} and instead adds a projection step, \ie,
\begin{subequations} \label{e-ipg2-mod}
\begin{align}
  z_{k} &= \argmin_{x \in \reals^n} \left\{ g_{i_k}(x) +
    \tnabla h_{i_k}(x_{k})^Tx + 
    \frac{1}{2t_k} \|x - x_{k}\|_2^2
  \right\} \\
  x_{k+1} &= \PC ( z_{k} ) ,
\end{align}
\end{subequations}
or equivalently, if we combine the two steps,
  \[
    x_{k+1} = \PC  \Bigl( \prox_{t_k g_{i_k}} \bigl( x_{k} -
    t_k \tnabla h_{i_k}(x_{k}) \bigr) \Bigr).
  \]
Note that this modified version of IPG2 is equivalent to IPG1 when
$h_i(x) = 0$ for all~$i$.

\subsection{The R-IPG1 and R-IPG2 Methods}

We are now ready to propose relaxed variants of IPG1 and the modified
IPG2 in \eqref{e-ipg2-mod}. The relaxed variant of IPG1, which we will call R-IPG1, depends on a
relaxation parameter $\rho \in (0,2)$, and it is defined as the iteration
\begin{subequations}\label{e-ripg1}
  \begin{align}
    \label{e-ripg1-1}
 \text{(R-IPG1)}  &&    w_{k} &= \prox_{t_k g_{i_k}} ( x_{k} ) &&\\
    \label{e-ripg1-2}
&&   z_{k} &= w_{k} - t_k \tnabla h_{i_k}(w_{k}) &&\\
    \label{e-ripg1-3}
  &&  x_{k+1} &= \PC \bigl( \rho z_{k}+ (1-\rho) x_{k} \bigr). &&
  \end{align}
\end{subequations}
Similarly, R-IPG2 refers to the relaxed variant of
\eqref{e-ipg2-mod}, and it is defined as
\begin{subequations} \label{e-ripg2}
  \begin{align}
\label{e-ripg2-1}
\text{(R-IPG2)} &&    w_{k} &= x_{k} - t_k \tnabla h_{i_k}(x_{k}) &&\\
\label{e-ripg2-2}
&&    z_{k} & = \prox_{t_k g_{i_k}} ( w_{k} ) &&\\
\label{e-ripg2-3}
  &&  x_{k+1} &= \PC \bigl( \rho z_{k} + (1-\rho) x_{k} \bigr). &&
  \end{align}
\end{subequations}
Notice that the relaxed algorithms \eqref{e-ripg1} and \eqref{e-ripg2}
are very similar, and they differ only in the order of the first two
updates at each iteration.

\begin{remark}
  It is easy to verify that the two relaxed methods produce the exact
  same sequence $\{x_k\}$ if either $h_i(x) = 0$ or $g_i(x) = 0$ for
  all $i$. In the latter case, both R-IPG1 and R-IPG2 reduce to a
  projected (sub)gradient method with step size $\rho t_k$, and this
  implies that the relaxation parameter is redundant when $g_i(x) = 0$
  for all $i$.
\end{remark}

\subsection{Convergence Results}
\label{subsec:convergence}

We now address the convergence properties of R-IPG1 and R-IPG2 using
cyclic control. Following the exposition in \cite{Ber:11}, we will make the following
assumptions about the functions $g_i$ and $h_i$ and their
(sub)gradients.
\begin{assumption}{\bf (R-IPG1)} \label{assumption-ripg1}
There exists a constant $c$ such that for all $k$,
  \begin{align}
  \max\bigl\{ \| \tnabla g_{i_k}( w_{k} ) \|_2 ,  \| \tnabla
  h_{i_k}(w_{k}) \|_2 \bigr\} \leq c
\end{align}
and for all $k$ that mark the beginning of a cycle, we have for all $j = 1,\ldots,m$,
\begin{align}
  \max\left\{ g_j(x_{k}) - g_j(w_{k+j-1}), h_j(x_{k}) - h_j(w_{k+j-1}) \right\}
  \leq c\, \| x_{k} - w_{k+j-1} \|_2 .
\end{align}
\end{assumption}
\begin{assumption}{\bf (R-IPG2)} \label{assumption-ripg2}
There exists a constant $c$ such that for all $k$,
  \begin{align}
  \max \bigl\{ \| \tnabla g_{i_k}(z_{k}) \|_2 , \| \tnabla
    h_{i_k}(x_{k}) \|_2 \bigr\} \leq c
\end{align}
and for all $k$ that mark the beginning of a cycle, we have for all $j=1,\ldots,m$,
\begin{gather}
  \max \left\{  g_j(x_{k}) - g_j(x_{k+j-1}) ,  h_j(x_{k}) -
    h_j(x_{k+j-1})  \right \} \leq c \,\| x_{k} - x_{k+j-1} \|_2 \\
  g_j(x_{k+j-1}) - g_j(z_{k+j-1}) \leq c \,\|x_{k+j-1}
  -z_{k+j-1} \|_2 .
\end{gather}
\end{assumption}

\begin{remark}
Assumptions \ref{assumption-ripg1} and \ref{assumption-ripg2} are
satisfied if, for example, all $g_i$ and $h_i$ are Lipschitz
continuous on $\reals^n$, or if the sequences $\{x_k\}$ and $\{w_k\}$
(in the case of R-IPG1) or $\{x_k\}$ and $\{z_k\}$ (in case of R-IPG2)
are bounded. See \cite{Ber:11} for further details.
\end{remark}

A key component of the convergence analysis is the following generalization of Proposition~3 in \cite{Ber:11}.
\begin{proposition} \label{prop-ripg-bound}
Let $\{x_k\}$ be a sequence generated by either
\eqref{e-ripg1} or \eqref{e-ripg2} with the index $i_k$ chosen according
to the cyclic rule $i_k = (k \mod{m}) +1$. Then, given a point $y \in
\mathcal C$ and a relaxation parameter $\rho \in [\delta, 2 -
\delta]$ for some $\delta > 0$,
\begin{align} \label{e-prop-ripg-bound}
  \| x_{k+m} - y \|_2^2 \leq \| x_{k} - y \|_2^2 - 2\rho t_k (f(x_{k}) -
  f(y)) + \beta \rho^2 t_k^2m^2c^2
\end{align}
where
  \[
    \beta = \left\{ \begin{array}{ll}
      4+ \frac{1-\rho + \alpha}{\rho m} & \quad \textrm{for R-IPG1~\eqref{e-ripg1}} \\[2mm]
      4 + \frac{4(1-\rho) + \alpha}{\rho m} & \quad \textrm{for R-IPG2~\eqref{e-ripg2}} ,
    \end{array} \right.
  \]
in which $\alpha$ is a constant defined as
\begin{align*}
  \alpha =
  \begin{cases}
    1/(2-\rho) & \delta \leq \rho \leq 3/2 \\
    4(1-\rho) & 3/2 < \rho \leq 2-\delta.
  \end{cases}
\end{align*}
\end{proposition}
\begin{proof}
  See Appendix \ref{app-proof-prop}.
\end{proof}
\begin{remark} If we let $\rho = 1$ in
  Proposition~\ref{prop-ripg-bound}, then we obtain $\beta = 4 + 1/m$
  for both of the relaxed methods.  The constant for IPG2 derived in
  \cite{Ber:11} is $\beta = 4 + 5/m$; this discrepancy arises because
  of an approximation $m^2 \approx m^2-m$ in the proof in
  \cite{Ber:11}, and without this approximation we obtain $\beta = 4 +
  1/m$ for both IPG1 and IPG2. Figure~\ref{fig-beta} shows the
  constant $\beta$ as a function of $\rho$ for both R-IPG1 and R-IPG2
  and different values of $m$.
\end{remark}
\begin{figure}
  \centering
  \subfloat[R-IPG1]{\begin{tikzpicture}
	\begin{semilogyaxis}[
          height=1.8in,
          width=2.4in,
          xmin=0,xmax=2,
          ymin=3.2,ymax=40, 
          ytick={4,8,16,32},
          yticklabels={4,8,16,32},
          xlabel=$\rho$,
          ylabel=$\beta$,
          ylabel style={yshift=-0.2in},
          legend style={draw=none,cells={anchor=west},anchor=north east, at={(0.95,0.95)},font=\small}
          ]
          
          \addplot[thick, dotted, color=red, mark=none, domain=1e-3:2-1e-3, samples=200] 
          {4 + (min(2,1./(2-x)) + 4*max(0,x-3/2) + (1-x))/(x*1)};
          \addlegendentry{$m=1$}
          
          \addplot[thick, dashed, color=blue, mark=none, domain=1e-3:2-1e-3, samples=200] 
          {4 + (min(2,1./(2-x)) + 4*max(0,x-3/2) + (1-x))/(x*4)};
          \addlegendentry{$m=4$}          

          \addplot[thick, color=black, mark=none, domain=1e-3:2-1e-3, samples=200] 
          {4 + (min(2,1./(2-x)) + 4*max(0,x-3/2) + (1-x))/(x*100)};
          \addlegendentry{$m=100$}

	\end{semilogyaxis}
\end{tikzpicture}}
  \quad
  \subfloat[R-IPG2]{\begin{tikzpicture}
	\begin{semilogyaxis}[
          height=1.8in,
          width=2.4in,
          xmin=0,xmax=2,
          ymin=3.2,ymax=40, 
          ytick={4,8,16,32},
          yticklabels={4,8,16,32},
          xlabel=$\rho$,
          ylabel=$\beta$,
          ylabel style={yshift=-0.2in},
          legend style={draw=none,cells={anchor=west},anchor=north east, at={(0.95,0.95)},font=\small}
          ]
          
          \addplot[thick, dotted, color=red, mark=none, domain=1e-3:2-1e-3, samples=200] 
          {4 + (min(2,1./(2-x)) + 4*max(0,x-3/2) + 4*(1-x))/(x*1)};
          \addlegendentry{$m=1$}
          
          \addplot[thick, dashed, color=blue, mark=none, domain=1e-3:2-1e-3, samples=200] 
          {4 + (min(2,1./(2-x)) + 4*max(0,x-3/2) + 4*(1-x))/(x*4)};
          \addlegendentry{$m=4$}          

          \addplot[thick, color=black, mark=none, domain=1e-3:2-1e-3, samples=200] 
          {4 + (min(2,1./(2-x)) + 4*max(0,x-3/2) + 4*(1-x))/(x*100)};
          \addlegendentry{$m=100$}

	\end{semilogyaxis}
\end{tikzpicture}}
  \caption{The two plots show the magnitude of the constant $\beta$ in
    the bound \eqref{e-prop-ripg-bound} as a function of $\rho$ for
    each of the relaxed incremental methods. Notice that, in both
    cases, $\beta \approx 4$ over a wide interval when $m$ is large.}
  \label{fig-beta}
\end{figure}
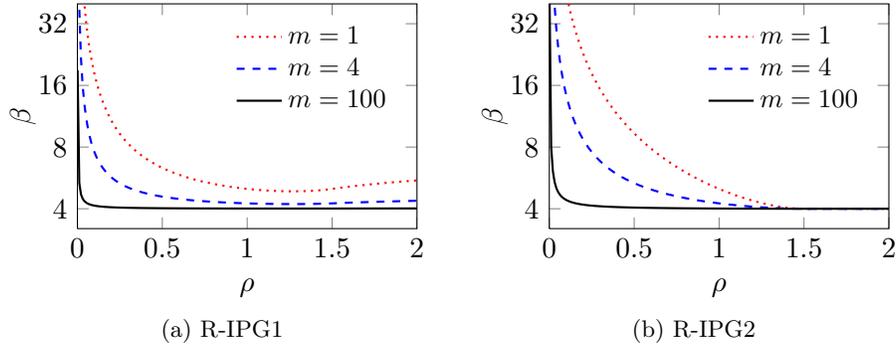

The following proposition summarizes the main convergence results for
problems where $f(x)$ is bounded below and using cyclic control.
\begin{proposition}\label{prop-conv} Let $\{x_k\}$ be a sequence generated by either
  \eqref{e-ripg1} or \eqref{e-ripg2}, and suppose $f(x)$ is bounded
  below ($f^\star > -\infty$). Then, using cyclic control we have an
  error bound for constant step size $t_k = t$
 \[ \liminf_{k\rightarrow \infty} \, f(x_k) = f^\star + \frac{\rho t \beta
     m^2c^2 }{2}\]
and exact convergence for a diminishing step-size rule that satisfies $\sum_{k=1}^\infty t_k =
  \infty$ and $ \lim_{k\rightarrow \infty} t_k = 0$, \ie,
\[  \liminf_{k\rightarrow \infty} \, f(x_k) = f^\star. \]
\end{proposition}
\begin{proof}
The bound \eqref{e-prop-ripg-bound} has the exact same form as the
bound in Proposition~3 in \cite{Ber:11} if we define a scaled
parameter $\tilde t_k = \rho t_k$. Since the convergence analysis in
\cite{Ber:11} is based on this bound, it also holds for the relaxed
variants of the methods; see Propositions 4 and 6 in \cite{Ber:11}.
  \end{proof}
\begin{remark}
  It is also possible to obtain similar bounds (in expectation) for
  R-IPG1 and R-IPG2 with randomized control. However, since the
  analysis is nearly identical to that in \cite{Ber:11}, we omit the
  details for the sake of brevity. Further details and efficiency
  estimates can be found in \cite{Ber:11}.
\end{remark}

The error bound for the above methods is $O(mc/\sqrt{\ell})$ where $\ell$ is
the number of cycles. We remind the reader that despite this poor
global error bound, the incremental methods often have fast initial
rate of convergence and may outperform nonincremental methods when
low-accuracy is acceptable. 

\section{ART Within the R-IPG Framework}
\label{sec:art}

This section shows an important application of the algorithmic
framework introduced in the previous section.
In particular, we demonstrate how specialized variants of the relaxed algorithms
lead to the well-known ART method and variants of this method.

\subsection{Relaxed ART Methods}
ART can be viewed both as an incremental gradient method and as an
incremental proximal method.  Specifically, if we let $g_i(x) = 0$ and
$h_i = \nicefrac{1}{2}(a_i^Tx - b_i)^2/ \|a_i\|_2^2$,
both R-IPG1 and R-IPG2 result in the incremental gradient iteration
\begin{align}\label{e-ripg-ig1}
  x_{k+1} = \PC \!\left( x_k - \rho t_k \, a_{i_k}
  (a_{i_k}^Tx_k - b_{i_k}) / \|a_{i_k}\|_2^2  \right) ,
\end{align}
which is equivalent to ART with relaxation parameter $\rho$ if we let $t_k = 1$.
Similarly, if we let $g_i(x) = I_{\mathcal{H}_i}(x)$ and $h_i = 0$ where
$
    \mathcal H_i = \{ x \in \reals^n \,|\, a_i^Tx - b_i = 0\} ,
$
then both R-IPG1 and R-IPG2 result in the following algorithm
\begin{align} \label{e-ripg-ip1}
  x_{k+1} = \PC \bigl(\rho  \proj_{\mathcal H_{i_k}}(x_k) + (1-\rho) x_k \bigr),
\end{align}
where the projection of $x_k$ on $\mathcal H_{i_k}$ is given by
$\proj_{\mathcal H_{i_k}}(x_k) = x_k - a_{i_k} (a_{i_k}^Tx_k-b_{i_k}) /
\|a_{i_k}\|_2^2$.
When we insert this relation into \eqref{e-ripg-ip1}, we once again obtain ART
with relaxation parameter~$\rho$. 

Note that although the iteration \eqref{e-ripg-ip1} is an incremental
proximal algorithm, the choice $g_i(x) = I_{\mathcal{H}_i}(x)$ does
not satisfy Assumptions \ref{assumption-ripg1} and
\ref{assumption-ripg2}.  In fact, the corresponding problem is a
convex feasibility problem that may or may not be feasible.  ART, 
however, is known to converge to the minimum norm solution if the system
$Ax=b$ is consistent (\ie, the feasibility problem is feasible), and
otherwise ART converges to a weighted least-squares solution provided
that a diminishing step-size sequence is used
\cite{CEG:83,JiW:03}. Similarly, using Proposition~\ref{prop-conv}, it
follows that the iteration \eqref{e-ripg-ig1} converges to a weighted
least-squares solution when a diminishing step-size sequence is
used. 

An alternative to the choice $g_i(x) = I_{\mathcal{H}_i}(x)$ is to
define $g_i(x) = \dist(x, \mathcal H_i)$ or $g_i(x) = \dist(x,
\mathcal H_i)^2$, and as we will see in the next section, this gives
rise to \emph{damped} ART-like algorithms.

\subsection{Damped ART}
\label{subsection:DART}

It is an interesting and useful fact that
there are many other possible choices of $g_{i}$ and $h_{i}$ that lead to
convergent incremental methods that are similar to ART\@.
For example, if we let $g_i(x) = \nicefrac{1}{2}(a_i^Tx - b_i)^2$ and $h_i(x) = 0$,
we obtain the following incremental proximal method
  \begin{align}
  \label{e-ripg-ip3}
    x_{k+1} = \PC \left( x_k - \rho
    \frac{a_{i_k}(a_{i_k}^Tx_k - b_{i_k})}{ \|a_{i_k}\|_2^2 + t_k^{-1}} \right).
  \end{align}
This iteration can be viewed as a \emph{damped} ART method where $t_k$
determines the damping at step $k$. Large values of $t_k$ correspond to a
small amount of damping, and it is easy to verify that in the limit,
if we let $t_k \rightarrow \infty$, the iteration is equivalent to ART\@.
This variant of ART is useful when some rows have very small but nonzero norm,
in which case the damping helps to suppress noise amplification---we illustrate
this with an example in \S\ref{subsec:advantage}.

It is also possible to derive generalized block methods based on the
relaxed incremental proximal gradient methods. Here we consider block
iterative methods for minimizing $\| Ax -b\|_2^2$.
Suppose we partition $A$ and $b$ into $p$ blocks of rows where $A_i \in
\reals^{m_i \times n}$ denotes the $i$th block of $A$ and $b_i \in
\reals^{m_i}$ denotes the $i$th block of $b$.
A variant of Elfving's block-Kaczmarz method \cite{Elf:80} then follows from iteration
\eqref{e-ripg-ip1} if we let $\mathcal B_i = \{ x\in \reals^n \,|\,
A_i x = b_i \}$. Since the projection of a point $x$ onto $\mathcal
B_i$ can be expressed as $\proj_{\mathcal B_i}(x) = x - A_i^\dagger (A_ix -
b_i)$, we can express the block Kaczmarz method as
\begin{align} \label{e-block-kaczmarz}
  x_{k+1} = x_k - \rho A_{i_k}^\dagger (A_{i_k}x - b_{i_k}).
\end{align}
Notice that like Kaczmarz's method, the block Kaczmarz method does not
include the parameter $t_k$.
If we instead let $g_i(x) = \nicefrac{1}{2} \|A_ix - b_i\|_2^2$, we obtain the following proximal operator
\begin{align} \label{e-block-art-prox}
  \prox_{t_k g_{i_k}}(x) &= ( I + t_k A_{i_k}^TA_{i_k})^{-1} (x+ t_k A_{i_k}^Tb_{i_k}) \\
  &= x - M_{i_k} (A_{i_k} x - b_{i_k}) ,
\end{align}
where $M_{i_k} = A_{i_k}^T(A_{i_k}A_{i_k}^T + t_k^{-1}I)^{-1}$.
From the limit definition
\cite{Alb:72}
\begin{align*}
  A^\dagger = \lim_{\delta \rightarrow 0} (A^TA + \delta^2 I)^{-1}A =
  \lim_{\delta \rightarrow 0} A^T(AA^T + \delta^2 I)^{-1},
\end{align*}
we immediately see that $M_{i_k} \rightarrow A_{i_k}^\dagger$ as
$t_k \rightarrow \infty$. We can therefore interpret
$M_{i_k}$ as a regularized or \emph{damped} pseudoinverse of $A_{i_k}$, and
hence the resulting incremental method is a block variant of
\eqref{e-ripg-ip3}.

We obtain yet another damped ART-like algorithm if we let
$h_i(x) = 0$ and either $g_i(x) = \dist(x, \mathcal H_i)$ or $g_i(x) =
\nicefrac{1}{2}\,\dist(x, \mathcal H_i)^2$. For example, if we let
$g_i(x) = \nicefrac{1}{2}\,\dist(x, \mathcal H_i)^2$, the proximal
operator associated with $g_{i_k}$ is given by
\begin{align*}
  \prox_{t_k g_{i_k}}(u) = (1-\theta_k) u + \theta_k \proj_{\mathcal H_{i_k}}(u) ,
\end{align*}
where $\theta_k = \min\bigl(1, t_k\|a_{i_k}\|_2/|a_{i_k}^Tu - b_{i_k}|\bigr)$; see, \eg,
\cite{CoP:11}. The resulting relaxed incremental
proximal iteration is of the form
\begin{align} \label{e-ripg-ip2}
  x_{k+1} =
  \begin{cases}
    \PC \left( x_k - \rho \displaystyle\frac{a_{i_k}(a_{i_k}^Tx_k-b_{i_k})}{\|a_{i_k}\|_2^2} \right)& |a_{i_k}^Tx_k-b_{i_k}| < t_k \|a_{i_k}\|_2 \\[1em]
    \PC\left( x_k - \rho t_k \displaystyle
    \frac{a_{i_k}}{\|a_{i_k}\|_2} \sgn(a_{i_k}^Tx_k - b_{i_k}) \right) & \text{otherwise}
  \end{cases}
\end{align}
if $\|a_{i_k}\|_2 > 0$, and otherwise $x_{k+1} = x_k$.

\subsection{Damped ART for Robust Regression}

It is well-known that the least squares objective $\|Ax-b\|_2^2$ yields a
maximum a posteriori (MAP) estimate when the noise is Gaussian.
However, $\ell_2$ data fitting is sensitive to outliers.
A more robust criterion is the $\ell_1$ norm objective
$\|Ax-b\|_1$ (which is also known as linear \emph{least absolute
  value} regression), and this yields a MAP estimate
when the noise follows a Laplace distribution \cite{Die:05}.
Minimizing the $\ell_1$ norm of the residuals instead of the squared
$\ell_2$ norm leads to another damped ART-like algorithm. Specifically, if we
let $g_i(x) = |a_i^Tx-b_i|$ and $h_i(x) =0$, then both R-IPG1 and R-IPG2
lead to the following update
\begin{align}
  \label{e-ripg-ip4}
  x_{k+1} =
  \begin{cases}
    \PC \left( x_k - \rho \displaystyle\frac{a_{i_k}(a_{i_k}^Tx_k-b_{i_k})}{\|a_{i_k}\|_2^2} \right)& |a_{i_k}^Tx_k-b_{i_k}| < t_k \|a_{i_k}\|_2^2 \\[1em]
        \PC \bigl( x_k - \rho t_k a_{i_k}\sgn(a_{i_k}^Tx_k-b_{i_k})\bigr) & \text{otherwise}.
  \end{cases}
\end{align}
To see this, consider the proximal operator
\begin{align*}
  \prox_{t_k g_{i_k}}(x_k)  =\argmin_{u \in \reals^n} \left\{ t_k\, | a_{i_k}^Tu - b_{i_k} | + \nicefrac{1}{2} \|u -x_k \|_2^2 \right\}.
\end{align*}
Clearly, $\prox_{t_k g_{i_k}}(x_k) = x_k$ if either $a_{i_k} = 0$ or
$a_{i_k}^Tx_k = b_{i_k}$, and otherwise the minimizer $u^\star$ is
attained on the line segment between $x_k$ and its projection
on the affine subspace $\{x\,|\, a_{i_k}^Tx = b_{i_k} \}$, \ie,
\begin{align}
  \label{e-ripg-4-prox-sol}
 u^\star = x_k - \theta^\star a_{i_k}(a_{i_k}^Tx_k - b_{i_k}) / \|a_{i_k}\|_2^2
\end{align}
for some $\theta^\star \in (0,1]$. Thus, using this parameterization of
$u^\star$, we may evaluate the proximal operator by computing
\begin{align*}
 \theta^\star &= \argmin_{\theta \in (0,1]} \left\{ t_k (1-\theta)|
  a_{i_k}^Tx_k - b_{i_k}| + \nicefrac{\theta^2}{2} (a_{i_k}^Tx_k -
  b_{i_k})^2/\|a_{i_k}\|_2^2 \right\} \\
  &= \min\left(1, t_k \|a_{i_k}\|_2^2 /|a_{i_k}^Tx_k - b_{i_k}|\right)
\end{align*}
and the minimizer $u^\star = \prox_{t_k g_{i_k}}(x_k)$ then follows from
\eqref{e-ripg-4-prox-sol}. Using this result in R-IPG1 or R-IPG2, we
obtain the iteration \eqref{e-ripg-ip4}.  This is very
similar to the iteration \eqref{e-ripg-ip2}, but the step-size rule is
based on $\|a_{i_k}\|_2$ instead of $\|a_{i_k}\|_2^2$.

The update \eqref{e-ripg-ip4} is simply a (relaxed) projection if the
magnitude of the residual $|a_{i_k}^Tx_k - b_{i_k}|$ is sufficiently
small. On the other hand, if $|a_{i_k}^Tx_k - b_{i_k}|$ is
sufficiently large, the step will be damped, and the parameter $t_k$ governs the
damping. Notice that like the method \eqref{e-ripg-ip3}, the
method \eqref{e-ripg-ip4} also reduces to ART if $t_k$ is sufficiently
large.

The Huber penalty $\phi_\mu(t)$, which is defined as
  \[
    \phi_\mu(t) =
    \begin{cases}
      \nicefrac{t^2}{(2\mu)} & |t| < \mu\\
      |t| - \nicefrac{\mu}{2} & |t|\geq \mu
    \end{cases}
  \]
where $\mu \geq 0$ is a parameter, can be viewed as a combination of the
$\ell_1$ and $\ell_2$ norms.
If we define $g_i(x) = \phi_\mu(a_i^Tx - b_i)$ and $h_i(x) = 0$ in R-IPG1
or R-IPG2, we obtain the following iteration
\begin{align}
  x_{k+1} =
  \begin{cases}
   \PC\left( x_k - \rho \displaystyle\frac{a_{i_k} (a_{i_k}^Tx_k -
     b_{i_k})}{\nicefrac{\mu}{t_k} + \|a_{i_k} \|_2^2} \right) &
   |a_{i_k}^Tx_k -b_{i_k}| < \mu + t_k \|a_{i_k}\|_2^2 \\[1em]
   \PC\bigl( x_k - \rho t_k a_{i_k} \sgn(a_{i_k}^Tx_k - b_{i_k}) \bigr)&  \text{otherwise.}
  \end{cases}
\end{align}
Note that this reduces to the algorithm \eqref{e-ripg-ip4} for
$\ell_1$ norm minimization when $\mu = 0$.

\section{Generalized Row-Action Methods with Regularization}
\label{sec:artreg}

We now turn to constrained \textit{regularized} least-squares problems of the form
\begin{align} \label{e-reg-ls}
\begin{array}{ll}
  \mbox{minimize}
  &  f(x) \equiv \nicefrac{1}{2} \sum_{i=1}^m (a_i^Tx - b_i)^2 +  \lambda \psi (x)\\[1mm]
  \mbox{subject to}
  & x \in \mathcal C ,
\end{array}
\end{align}
where $\lambda > 0$ is a regularization parameter.
We will assume that the regularization function $\psi (x)$ is convex
(but not necessarily smooth).
Notice that this problem is of the
form \eqref{e-opt-sum-of-functions} if, for example, we let $g_i(x) =
\nicefrac{1}{2} (a_i^Tx - b_i)^2$ and $h_i(x) = \nicefrac{\lambda}{m} \psi (x)$.
There are obviously many other ways to express \eqref{e-reg-ls} as a
problem of the form \eqref{e-opt-sum-of-functions}, and these give
rise to a family of relaxed incremental proximal gradient algorithms
for the problem \eqref{e-reg-ls}. Note that although the resulting
algorithms may appear to be somewhat similar, they may behave
very differently in practice.

A straightforward way to construct an incremental method for the
regularized least-squares problem \eqref{e-reg-ls} is to define $m$ components
\begin{align}
  \label{e-reg-ls-components-1}
  g_i(x) = \nicefrac{1}{2} (a_i^Tx - b_i)^2, \quad h_i(x) =
  \nicefrac{\lambda}{m} \psi(x), \quad i=1,\ldots,m.
\end{align}
This choice results in algorithms that alternate
between a small (sub)gradient step for the regularization term, a
damped projection on a hyperplane defined by one of the equations
$a_i^Tx - b_i = 0$, and a projection on $\mathcal C$.

We obtain a different pair of algorithms if we instead define $m+s$ components
\begin{subequations}\label{e-reg-ls-components-2}
  \begin{align}
  \label{e-reg-ls-components-2-a}
   g_i(x) &= \nicefrac{1}{2} (a_i^Tx-b_i)^2, & h_i(x) &= 0, \quad i=1,\ldots,m,\\
  \label{e-reg-ls-components-2-b}
  g_{m+i}(x) &= \nicefrac{\lambda}{s} \psi(x), &  h_{m+i}(x) &= 0, \quad i=1,\ldots,s,
\end{align}
or alternatively, instead of \eqref{e-reg-ls-components-2-b},
\begin{align}
  \label{e-reg-ls-components-2-c}
  g_{m+i}(x) &=0, & h_{m+i}(x)& = \nicefrac{\lambda}{s} \psi(x), \quad i=1,\ldots,s.
\end{align}
\end{subequations}
With cyclic component selection, this results in ART-like algorithms
where each cycle consists of $m$ damped projections, followed by
either $s$ proximal steps associated \eqref{e-reg-ls-components-2-b} or $s$
(sub)gradient steps associated with \eqref{e-reg-ls-components-2-c}.

These algorithms are similar to the ``superiorization method'' of
Censor et al.\ \cite{CDH:10} which is also an ART-like method where
each cycle consists of a complete ART-cycle, followed by a
``correction step'' which is referred to as a perturbation. The
correction step can be a (sub)gradient step, and if the set $\mathcal
D = \{ x \,|\, x\in \mathcal C, \, Ax = b\}$ is nonempty, then the
superiorization method converges to a point in $\mathcal D$ provided
that the norm of the correction at iteration $k$ goes to zero as
$k\rightarrow\infty$.
If the correction is a negative subgradient of a
regularization function $\psi(x)$, the method tends to converge to
points in $\mathcal D$ for which the regularization term $\psi(x)$ is
small compared to what can be achieved with plain ART\@.  In contrast,
our approach leads to methods that always converge to a solution of the
regularized least-squares problem \eqref{e-reg-ls}, provided that a
diminishing step-size sequence is used.

\subsection{Least-Squares with Total Variation Regularization}

Total Variation (TV) regularization \cite{ROF:92} is popular in imaging
because of its ability to suppress noise while preserving edges. In
the discrete setting, the TV seminorm of a vector representation $x
\in \reals^n$ of a $d$ dimensional image $X$ can be expressed as a
mixed $\ell_{1,2}$ norm
\begin{align}
  \label{e-tv-seminorm} \textstyle
    \psi(x) = \| Dx \|_{1,2} \equiv \sum_{i=1}^n \|D_i x\|_2 ,
\end{align}
 where $D_i$ is a $d
 \times n$ matrix such that $D_ix$ is a finite-difference
 approximation of the gradient at the $i$th pixel ($d=2$) or voxel
 ($d=3$), and $D_i$ is the $i$th block-row of $D \in \reals^{dn \times
   n}$. Note that the
 definition of $D_i$ depends on both boundary conditions and the
 finite-difference approximation of the gradient.

 The TV seminorm \eqref{e-tv-seminorm} is not everywhere differentiable, but it
 is easy to compute a subgradient using the chain rule and the following property
\begin{align}
  \partial \| x \|_2 =
  \begin{cases}
     \{ x / \| x\|_2 \} &  \|x\|_2 > 0 \\
     \{ y\,|\, \|y\|_2 \leq 1 \} &  \text{otherwise}.
  \end{cases}
\end{align}
Numerically we may compute a subgradient as
  \[
    \tnabla \psi (x) \approx D^T\diag(w_1I,\ldots, w_n I)^{-1} Dx ,
  \]
where $w_i = \max\{\tau, \|D_i x\|_2\}$ or $w_i = \|D_ix\|_2 + \tau
$ for some small $\tau > 0$ to avoid dividing by zero or a number
close to zero. Note that with the choice $w_i =
\max\{\tau, \|D_i x\|_2\}$, the above subgradient approximation can
be interpreted as the gradient of a smoothed version of the TV seminorm
  \[ \textstyle
    \psi_{\tau} (x) = \sum_{i=1}^m \phi(\| D_i x\|_2) ,
  \]
where $\phi_{\tau} : \reals \to \reals$ is the scaled Huber penalty
\[\phi_{\tau}(u) =
\begin{cases}
  (u)^2/(2\tau) & |u| \leq \tau \\
  |u| - \tau/2 & \text{otherwise}
\end{cases}
\]
which is once differentiable.

If we define $g_i$ and $h_i$ as in \eqref{e-reg-ls-components-2-a} and
\eqref{e-reg-ls-components-2-c}, the resulting incremental methods are
similar to the ASD-POCS method of Sidky et al.\ \cite{SiP:08}.
This method seeks a solution to the constrained TV-minimization problem
\begin{align}\label{e-reg-ls-constrained}
  \begin{array}{ll}
    \mbox{minimize}
    & \|Dx \|_{1,2} \\[1mm]
    \mbox{subject to}
    & \| Ax - b\|_2 \leq \gamma \\
    & x \in  \mathcal C ,
  \end{array}
\end{align}
where $\gamma$ is a constant. The problem \eqref{e-reg-ls-constrained}
is equivalent to the problem \eqref{e-reg-ls} (with $\psi(x) =
\|Dx\|_{1,2}$) in the sense that for each $\lambda > 0$, there exists
a constant $\gamma > 0$ such that both problems have the same set of
minimizers. Each iteration of the ASD-POCS method consists of a
complete ART cycle and a projection on $\mathcal C$, followed by $s$
subgradient steps based on subgradients of the TV seminorm. The method
is adaptive in the sense that it adjusts both the step size used in the
ART cycle and the step size used for the subgradient steps at each
iteration. As a consequence, ASD-POCS does not necessarily converge to
a solution to the problem \eqref{e-reg-ls-constrained}, but in
practice it often produces a feasible $x$ with low TV quite fast.

As mentioned in the beginning of this section, we obtain another pair
of incremental methods if we define the functions $g_i(x)$ and
$h_i(x)$ as in \eqref{e-reg-ls-components-2-a} and
\eqref{e-reg-ls-components-2-b}. The resulting methods do not require
a subgradient of the TV seminorm, but instead we need to evaluate the
proximal operators $\prox_{t_k g_{m+i}}(x)$, $i=1,\ldots,s$. If we
choose $s = 1$ and $g_{m+1} = \lambda \|Dx\|_{1,2}$, this amounts to
solving an unconstrained TV denoising problem, \ie,
\[ \prox_{t_k g_{m+1}}(x) = \argmin_{u\in \reals^n} \left\{ t_k \lambda
  \|Du\|_{1,2} + \nicefrac{1}{2} \|u - x \|_2^2 \right\}. \]
This is a strongly convex optimization problem that can be solved
efficiently with \eg\ FISTA \cite{BeT:09} or NESTA \cite{BBC:11}.  The
resulting relaxed incremental proximal algorithms resemble ART in that
every cycle involves damped projections onto the $m$ hyperplanes
defined by the rows of $A$, and in addition, every cycle also includes
a denoising step.

\subsection{Scaled Least-Squares with Regularization}

As a second example,
we consider the scaled regularized least-squares problem
\begin{align}\label{e-reg-ls-scaled}
\begin{array}{ll}
  \mbox{minimize}
  & \nicefrac{1}{2} \| M(A T u - b)\|_2^2 + \lambda \psi(Tu) \\[1mm]
  \mbox{subject to}
  & Tu \in \mathcal C
\end{array}
\end{align}
with variable $u \in \reals^n$ and where
$M \in \reals^{m \times m}$ and $T \in \reals^{n\times n}$ are nonsingular.
This problem is directly related to the SIRT algebraic iterative methods;
see \cite{EHN:12} for different choices of $M$ and $T$.
The problem \eqref{e-reg-ls} is equivalent to
\eqref{e-reg-ls-scaled} when $M=I$, and this follows by making
a change of variables $x = Tu$ in \eqref{e-reg-ls-scaled}.Thus, we
may view $T$ as a preconditioner.
In practice the matrix $T$ should be chosen such that  projections on the set
$\mathcal{C}_T = \{ u\,|\, Tu \in \mathcal C\}$ are easy to compute, \ie,
  \[
    \proj_{\mathcal{C}_T}(\hat u) = \argmin_{Tu \in \mathcal C} \| u -
    \hat u \|_2^2 = T^{-1} \argmin_{v \in \mathcal C} \| T^{-1}(v - T \hat u)\|_2^2
  \]
should be cheap to evaluate. This is the case if, for example,  $\mathcal C$ is a ``box'' of the
form $\{ x\,|\, c \leq x \leq d\}$ with $c,d \in \reals^n$ and $T$ is
diagonal and positive; in this case we have
$\proj_{\mathcal{C}_T}(\hat u) = T^{-1} \proj_{\mathcal C}(T\hat u)$.

The problem \eqref{e-reg-ls-scaled} can be used to derive
``preconditioned'' variants of R-IPG1 and R-IPG2 for the problem
\eqref{e-reg-ls}. To demonstrate this, we let $M = I$ and define
$g_i(u) = \nicefrac{1}{2} (a_i^T T u - b_i)^2$ and $h_i(u) =
\nicefrac{\lambda}{m} \psi(Tu)$ for $i=1,\ldots, m$. The
R-IPG1 updates for the scaled problem \eqref{e-reg-ls-scaled} can then
be expressed as
\begin{align*}
  w_k &= u_k - \frac{T^Ta_{i_k}(a_{i_k}^TTu_k - b_{i_k})}{ \|T
  a_{i_k}\|_2^2 + t_k^{-1}} \\
  z_k &= w_k - \frac{t_k \lambda}{m} T^T \tnabla \psi(Tw_k) \\
  u_{k+1} &= \proj_{\mathcal{C}_T}( \rho z_k + (1-\rho)  u_k) ,
\end{align*}
and after a change of variables ($x_k = Tu_k$, $\tilde w_k = T w_k$
and $\tilde z_k = T z_k$), we obtain the following ``preconditioned'' R-IPG1 method for the problem \eqref{e-reg-ls}
\begin{subequations}
  \begin{align}
  \tilde w_k &= x_k - \frac{T T^T a_{i_k}  (a_{i_k}^Tx_k - b_{i_k}) }{\| Ta_{i_k} \|_2^2+t_k^{-1}  }\\
  \tilde z_k &= \tilde w_k - \frac{t_k \lambda}{m} TT^T \tnabla \psi(\tilde w_k) \\
  \label{e-scaled-ipg-x}
  x_{k+1} &=  \proj_{\mathcal{C}_T} (\rho T^{-1}\tilde z_k +(1-\rho) T^{-1}x_k).
\end{align}
\end{subequations}
Although this method converges to a minimizer of \eqref{e-reg-ls} for
any nonsingular $T$ (provided that a diminishing step-size
sequence is used), the initial rate of convergence may vary in
practice. Moreover, if $T$ is diagonal, the projection
\eqref{e-scaled-ipg-x} can be expressed as
\begin{align}
  x_{k+1} = T^{-1} \PC(\rho z_k + (1-\rho) x_k).   
\end{align}
 A reasonable preconditioning strategy may be to define $T$
such that all the columns of $AT$ have unit $\ell_p$-norm.

\section{Numerical Results}
\label{sec-numres}

The numerical experiments described in this section were carried out
in \Matlab, and we used the package AIR Tools \cite{HaS:12} to generate the
$N\times N$ Shepp-Logan phantom represented by~$x^{\mathrm{exact}}$ and the
sparse matrix $A$.
The underlying model is a parallel-beam tomography problem with $p$ projections,
each involving $r$ rays; hence $A$ is $pr \times N^2$.

To avoid what is sometimes referred to as an ``inverse crime'' \cite{MuS:12},
we generate the measurement vector $b$ as follows.
First, using a larger number of rays $\bar{r}$ and a finer grid with
$\overline{N} \times \overline{N}$ pixels, we generate a noise-free sinogram
$\overline{B} \in \reals^{\bar{r} \times p}$;
in the experiments we use $\overline{N} = \mathrm{round}(\sqrt{3} N)$ and
$\bar{r} = \mathrm{round}(\sqrt{2}r)$.
We then use interpolation to compute a noise-free sinogram
$B \in \reals^{r \times p}$ that consists of $r$ samples for each projection
angle, and finally we obtain the noisy measurement data $b$ as
  \[
    b = b^{\mathrm{exact}} +  e , \qquad b^{\mathrm{exact}} = \mathrm{vec}(B) .
  \]
Here $e$ is a normally distributed noise vector with elements
$e_i \sim \mathcal{N}(0,\sigma^2)$ and
$\sigma$ is chosen such that $\| e \|_2 / \| b^{\mathrm{exact}} \|_2 = \eta$,
where we specify the noise level~$\eta$.

\subsection{The Advantage of Damped ART}
\label{subsec:advantage}

\begin{figure}
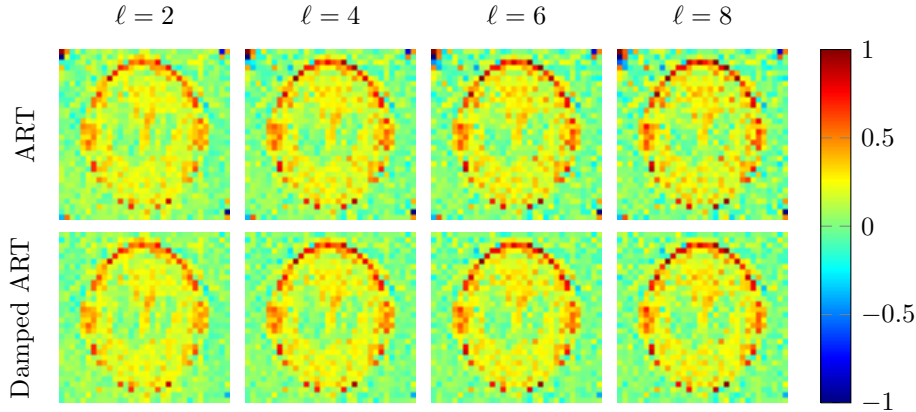

  \centering
  \input dart.tex
  \caption{\label{fig:DART} Four iterations of standard and damped ART methods for
  an example with noisy data; the damping $t_k^{-1} = 0.1\, \max_i \| a_i \|_2^2$
  suppresses the noise in the corners of the image
  without affecting the central part.}
\end{figure}

Our first example illustrates the use of the damped ART method
\eqref{e-ripg-ip3} from \S\ref{subsection:DART}.  When the noise level
$\eta$ is high, the reconstructed image tends to have large errors in
the corners pixels. These pixels correspond to rows of $A$ that have
small norm~$\| a_i \|_2$, and once such an error has occurred it stays
during the following iterations.

A simple and adaptive way to suppress these errors is to use damped
ART where we set the parameter $t_k$ in \eqref{e-ripg-ip3} to a
constant value chosen such that only updates for rows with small norm
are affected.  Figure \ref{fig:DART} illustrates this for an example
for the unrelaxed method ($\rho=1$) with $\eta = 0.08$, $N=32$, $r=32$, and
$p=36$.  The top row shows cycles $\ell=2,4,6,8$ for the standard ART
method, and the bottom row shows the same iterates for damped ART with
the choice $t_k^{-1} = 0.1\, \max_i \| a_i \|_2^2$.  The damping
clearly suppresses the noise in the corners of the image without
affecting the central part, and the results are not sensitive to the
factor (here, chosen as 0.1).

\subsection{Damped ART with Relaxation}
In the next experiment, we investigate the role of the relaxation
parameter $\rho$ and the parameter
$t_k$ for the damped ART method \eqref{e-ripg-ip3}.
We use a larger test problem with $N = 256$, $r = 362$, $p=120$, and
$\eta = 0.02$.
With this geometry, the average row norm $\|a_i\|_2$ is of the order 10.

Each of the plots in Figure~\ref{fig:ex_art_prox} shows the norm of the relative
error for different values of $\rho$ and with a fixed $t_k$.
Observe that when $t_k$ is small, the best performance is achieved with
overrelaxation (\ie, $1< \rho <2$) whereas when $t_k$ is large,
underrelaxation ($0< \rho < 1$) yields the best result.
Note also that with a large $t_k$, the relaxation parameter has a strong influence
on best iterate in terms of the minimum error.
In particular, the unrelaxed method ($\rho =1$) 
is poor for both $t_k=0.1$ and $t_k = 1.0$, but for $t_k =0.001$ its performance
is similar to that of the overrelaxed methods.
Finally recall that the
damped ART method is equivalent to ART if we let $t_k \rightarrow
\infty$, and in this example, the damped ART method is practically
indistinguishable from ART for $t_k \geq 1.0$.

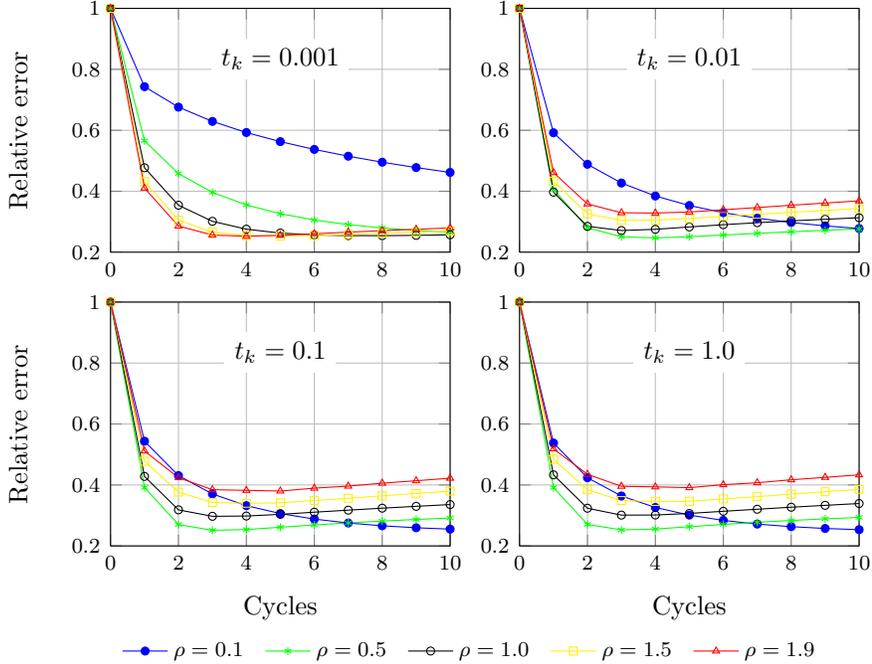
\begin{figure}
  \centering
  \begin{tikzpicture}
 
  \begin{axis}[
    name=plot1,
    font=\scriptsize,
    title={\small $t_k = 0.001$},
    title style={yshift=-3em,fill=white},
    ylabel={\small Relative error},
    ylabel style={overlay},
    xmin=0,
    xmax=10,
    ymin=0.2,
    ymax=1e0,
    width=2.4in,
    height=1.9in,
    grid=both
    ]

    \addplot [color=blue,mark=*,mark size=1.5pt] coordinates {
(0.0000e+00,1.0000e+00)(1.0000e+00,7.4305e-01)(2.0000e+00,6.7600e-01)(3.0000e+00,6.2909e-01)(4.0000e+00,5.9265e-01)(5.0000e+00,5.6272e-01)(6.0000e+00,5.3721e-01)(7.0000e+00,5.1497e-01)(8.0000e+00,4.9529e-01)(9.0000e+00,4.7769e-01)(1.0000e+01,4.6180e-01)
    };
    \addplot [color=green,mark=asterisk,mark size=1.5pt] coordinates {
(0.0000e+00,1.0000e+00)(1.0000e+00,5.6560e-01)(2.0000e+00,4.5818e-01)(3.0000e+00,3.9633e-01)(4.0000e+00,3.5518e-01)(5.0000e+00,3.2637e-01)(6.0000e+00,3.0567e-01)(7.0000e+00,2.9060e-01)(8.0000e+00,2.7958e-01)(9.0000e+00,2.7153e-01)(1.0000e+01,2.6570e-01)    
    };
    \addplot [color=black,mark=o,mark size=1.5pt] coordinates { (0.0000e+00,1.0000e+00)(1.0000e+00,4.7701e-01)(2.0000e+00,3.5412e-01)(3.0000e+00,3.0132e-01)(4.0000e+00,2.7553e-01)(5.0000e+00,2.6278e-01)(6.0000e+00,2.5676e-01)(7.0000e+00,2.5447e-01)(8.0000e+00,2.5429e-01)(9.0000e+00,2.5532e-01)(1.0000e+01,2.5704e-01)
    };
    \addplot [color=yellow,mark=square,mark size=1.5pt] coordinates {
(0.0000e+00,1.0000e+00)(1.0000e+00,4.3156e-01)(2.0000e+00,3.0539e-01)(3.0000e+00,2.6615e-01)(4.0000e+00,2.5443e-01)(5.0000e+00,2.5280e-01)(6.0000e+00,2.5481e-01)(7.0000e+00,2.5817e-01)(8.0000e+00,2.6196e-01)(9.0000e+00,2.6583e-01)(1.0000e+01,2.6958e-01)
    };
    \addplot [color=red,mark=triangle,mark size=1.5pt] coordinates { (0.0000e+00,1.0000e+00)(1.0000e+00,4.1037e-01)(2.0000e+00,2.8573e-01)(3.0000e+00,2.5634e-01)(4.0000e+00,2.5219e-01)(5.0000e+00,2.5549e-01)(6.0000e+00,2.6052e-01)(7.0000e+00,2.6575e-01)(8.0000e+00,2.7072e-01)(9.0000e+00,2.7532e-01)(1.0000e+01,2.7957e-01)
    };
  \end{axis}

  \begin{axis}[
    name=plot2,
    at=(plot1.right of south east), anchor=left of south west,
    font=\scriptsize, 
    title={\small $t_k = 0.01$},
    title style={yshift=-3em,fill=white},
    xmin=0,
    xmax=10,
    ymin=0.2,
    ymax=1e0,
    width=2.4in,
    height=1.9in,
    grid=both
    ]

    \addplot [color=blue,mark=*,mark size=1.5pt] coordinates {
(0.0000e+00,1.0000e+00)(1.0000e+00,5.9183e-01)(2.0000e+00,4.8842e-01)(3.0000e+00,4.2696e-01)(4.0000e+00,3.8443e-01)(5.0000e+00,3.5326e-01)(6.0000e+00,3.2970e-01)(7.0000e+00,3.1156e-01)(8.0000e+00,2.9745e-01)(9.0000e+00,2.8638e-01)(1.0000e+01,2.7768e-01)
    };
    \addplot [color=green,mark=asterisk,mark size=1.5pt] coordinates {
(0.0000e+00,1.0000e+00)(1.0000e+00,4.0604e-01)(2.0000e+00,2.8007e-01)(3.0000e+00,2.5074e-01)(4.0000e+00,2.4696e-01)(5.0000e+00,2.5078e-01)(6.0000e+00,2.5624e-01)(7.0000e+00,2.6182e-01)(8.0000e+00,2.6709e-01)(9.0000e+00,2.7194e-01)(1.0000e+01,2.7640e-01)
    };
    \addplot [color=black,mark=o,mark size=1.5pt] coordinates { 
(0.0000e+00,1.0000e+00)(1.0000e+00,3.9697e-01)(2.0000e+00,2.8487e-01)(3.0000e+00,2.7114e-01)(4.0000e+00,2.7513e-01)(5.0000e+00,2.8282e-01)(6.0000e+00,2.9023e-01)(7.0000e+00,2.9680e-01)(8.0000e+00,3.0272e-01)(9.0000e+00,3.0808e-01)(1.0000e+01,3.1308e-01)
    };
    \addplot [color=yellow,mark=square,mark size=1.5pt] coordinates {
(0.0000e+00,1.0000e+00)(1.0000e+00,4.3245e-01)(2.0000e+00,3.2602e-01)(3.0000e+00,3.0446e-01)(4.0000e+00,3.0536e-01)(5.0000e+00,3.1058e-01)(6.0000e+00,3.1760e-01)(7.0000e+00,3.2429e-01)(8.0000e+00,3.3093e-01)(9.0000e+00,3.3715e-01)(1.0000e+01,3.4327e-01)
    };
    \addplot [color=red,mark=triangle,mark size=1.5pt] coordinates { 
(0.0000e+00,1.0000e+00)(1.0000e+00,4.6129e-01)(2.0000e+00,3.5837e-01)(3.0000e+00,3.2951e-01)(4.0000e+00,3.2853e-01)(5.0000e+00,3.3183e-01)(6.0000e+00,3.3933e-01)(7.0000e+00,3.4645e-01)(8.0000e+00,3.5412e-01)(9.0000e+00,3.6128e-01)(1.0000e+01,3.6840e-01)
    };
  \end{axis}

 \begin{axis}[
    name=plot4,
    at=(plot2.below south west), anchor=above north west,
    font=\scriptsize, 
    title={\small $t_k = 1.0$},
    title style={yshift=-3em,fill=white},
    xlabel={\small Cycles}, 
    xmin=0,
    xmax=10,
    ymin=0.2,
    ymax=1e0,
    width=2.4in,
    height=1.9in,
    grid=both
    ]

    \addplot [color=blue,mark=*,mark size=1.5pt] coordinates {
(0.0000e+00,1.0000e+00)(1.0000e+00,5.3760e-01)(2.0000e+00,4.2375e-01)(3.0000e+00,3.6387e-01)(4.0000e+00,3.2626e-01)(5.0000e+00,3.0111e-01)(6.0000e+00,2.8378e-01)(7.0000e+00,2.7167e-01)(8.0000e+00,2.6318e-01)(9.0000e+00,2.5730e-01)(1.0000e+01,2.5332e-01)
    };
    \addplot [color=green,mark=asterisk,mark size=1.5pt] coordinates {
(0.0000e+00,1.0000e+00)(1.0000e+00,3.9142e-01)(2.0000e+00,2.7064e-01)(3.0000e+00,2.5260e-01)(4.0000e+00,2.5537e-01)(5.0000e+00,2.6309e-01)(6.0000e+00,2.7067e-01)(7.0000e+00,2.7746e-01)(8.0000e+00,2.8347e-01)(9.0000e+00,2.8882e-01)(1.0000e+01,2.9364e-01)
    };
    \addplot [color=black,mark=o,mark size=1.5pt] coordinates { 
(0.0000e+00,1.0000e+00)(1.0000e+00,4.3312e-01)(2.0000e+00,3.2366e-01)(3.0000e+00,3.0104e-01)(4.0000e+00,3.0142e-01)(5.0000e+00,3.0661e-01)(6.0000e+00,3.1364e-01)(7.0000e+00,3.2036e-01)(8.0000e+00,3.2699e-01)(9.0000e+00,3.3314e-01)(1.0000e+01,3.3915e-01)
    };
    \addplot [color=yellow,mark=square,mark size=1.5pt] coordinates {
(0.0000e+00,1.0000e+00)(1.0000e+00,4.8421e-01)(2.0000e+00,3.8434e-01)(3.0000e+00,3.4887e-01)(4.0000e+00,3.4586e-01)(5.0000e+00,3.4654e-01)(6.0000e+00,3.5437e-01)(7.0000e+00,3.6157e-01)(8.0000e+00,3.7002e-01)(9.0000e+00,3.7785e-01)(1.0000e+01,3.8556e-01)
    };
    \addplot [color=red,mark=triangle,mark size=1.5pt] coordinates { 
(0.0000e+00,1.0000e+00)(1.0000e+00,5.1845e-01)(2.0000e+00,4.3503e-01)(3.0000e+00,3.9566e-01)(4.0000e+00,3.9383e-01)(5.0000e+00,3.9116e-01)(6.0000e+00,4.0094e-01)(7.0000e+00,4.0725e-01)(8.0000e+00,4.1704e-01)(9.0000e+00,4.2496e-01)(1.0000e+01,4.3329e-01)
    };

  \end{axis}

 \begin{axis}[
    name=plot3,
    at=(plot4.left of south west), anchor=right of south east,
    font=\scriptsize, 
    title={\small $t_k = 0.1$},
    title style={yshift=-3em,fill=white},
    ylabel={\small Relative error},
    ylabel style={overlay},
    xlabel={\small Cycles}, 
    xmin=0,
    xmax=10,
    ymin=0.2,
    ymax=1e0,
    width=2.4in,
    height=1.9in,
    grid=both,
    legend style={legend columns=5,font=\scriptsize, draw=none},
    legend to name=ex_art_prox_legend
    ]

    \addplot [color=blue,mark=*,mark size=1.5pt] coordinates {
(0.0000e+00,1.0000e+00)(1.0000e+00,5.4355e-01)(2.0000e+00,4.3082e-01)(3.0000e+00,3.7053e-01)(4.0000e+00,3.3216e-01)(5.0000e+00,3.0619e-01)(6.0000e+00,2.8806e-01)(7.0000e+00,2.7522e-01)(8.0000e+00,2.6608e-01)(9.0000e+00,2.5962e-01)(1.0000e+01,2.5513e-01)
    };
    \addlegendentry[text width=0.45in]{$\rho=0.1$}

    \addplot [color=green,mark=asterisk,mark size=1.5pt] coordinates {
(0.0000e+00,1.0000e+00)(1.0000e+00,3.9155e-01)(2.0000e+00,2.7016e-01)(3.0000e+00,2.5135e-01)(4.0000e+00,2.5378e-01)(5.0000e+00,2.6131e-01)(6.0000e+00,2.6880e-01)(7.0000e+00,2.7553e-01)(8.0000e+00,2.8150e-01)(9.0000e+00,2.8681e-01)(1.0000e+01,2.9161e-01)
    };
    \addlegendentry[text width=0.45in]{$\rho=0.5$}

    \addplot [color=black,mark=o,mark size=1.5pt] coordinates { 
(0.0000e+00,1.0000e+00)(1.0000e+00,4.2802e-01)(2.0000e+00,3.1835e-01)(3.0000e+00,2.9718e-01)(4.0000e+00,2.9808e-01)(5.0000e+00,3.0364e-01)(6.0000e+00,3.1068e-01)(7.0000e+00,3.1736e-01)(8.0000e+00,3.2387e-01)(9.0000e+00,3.2990e-01)(1.0000e+01,3.3575e-01)
    };
    \addlegendentry[text width=0.45in]{$\rho=1.0$}

    \addplot [color=yellow,mark=square,mark size=1.5pt] coordinates {
(0.0000e+00,1.0000e+00)(1.0000e+00,4.7800e-01)(2.0000e+00,3.7661e-01)(3.0000e+00,3.4265e-01)(4.0000e+00,3.4002e-01)(5.0000e+00,3.4146e-01)(6.0000e+00,3.4912e-01)(7.0000e+00,3.5629e-01)(8.0000e+00,3.6451e-01)(9.0000e+00,3.7216e-01)(1.0000e+01,3.7971e-01)
    };
    \addlegendentry[text width=0.45in]{$\rho=1.5$}

    \addplot [color=red,mark=triangle,mark size=1.5pt] coordinates { 
(0.0000e+00,1.0000e+00)(1.0000e+00,5.1153e-01)(2.0000e+00,4.2393e-01)(3.0000e+00,3.8450e-01)(4.0000e+00,3.8225e-01)(5.0000e+00,3.8011e-01)(6.0000e+00,3.8955e-01)(7.0000e+00,3.9625e-01)(8.0000e+00,4.0582e-01)(9.0000e+00,4.1392e-01)(1.0000e+01,4.2217e-01)
    };
    \addlegendentry{$\rho=1.9$}

  \end{axis}
\end{tikzpicture}

\begin{tikzpicture}
 \ref{ex_art_prox_legend}    
\end{tikzpicture}
  \caption{The relative error $\|x_{\ell m} -
    x^{\mathrm{exact}}\|_2/\|x^{\mathrm{exact}}\|_2$ for damped ART
    with relaxation,
    after $\ell=0,1,2,\ldots,10$ full cycles, with different values of
    the relaxation parameter $\rho$ and a fixed parameter $t_k$.}
  \label{fig:ex_art_prox}
\end{figure}

\subsection{Generalized ART with Regularization}

In our last experiment, we consider the TV-regularized reconstruction
problem \eqref{e-reg-ls-constrained} with nonnegativity constraints,
\ie, $\mathcal C = \{x \,|\, x \geq 0 \}$.  We use the problem
parameters $N=512$, $r = 724$, $p = 60$, and $\eta = 0.01$, and we
split $A$ into $p$ blocks $A_1,\ldots, A_p$ (one for each parallel
projection) and define
  \[
    g_i(x) = \nicefrac{1}{2} \|A_ix - b_i\|_2^2, \quad h_i(x) =
    \nicefrac{\lambda}{p} \psi(x), \qquad i = 1,\ldots,p.
  \]
Here $\psi(x)$
is the TV seminorm defined in \eqref{e-tv-seminorm}, and $\lambda > 0$
is the regularization parameter.  The proximal operator associated
with $g_i(x)$ is given in \eqref{e-block-art-prox}. In this example, the matrix
$A_iA_i^T + t_k^{-1}I$ is tridiagonal because of the parallel-beam geometry,
and hence the proximal operator can be evaluated efficiently.

To establish a ground truth for the purpose of evaluating the quality
of the reconstructions, we first solve the TV-regularized
least-squares problem for a number of different regularization
parameters using the primal--dual first-order method of Chambolle and
Pock \cite{ChP:11}. We obtain the best result with $\lambda^\star
\approx 18$ for which the norm of the relative error is approximately
$0.14$. We then solve the TV-regularized problem with R-IPG1 for
different values of $\rho$, using the regularization parameter
$\lambda^\star$ and cyclic control.
Furthermore, we use a diminishing step-size sequence
in which the parameter $t_k$ is fixed throughout a complete cycle,
\ie, $t_k = t_0 \lceil \nicefrac{k}{p} \rceil^{-1}$ for $k=1,2,\ldots$,
where $t_0$ is the initial value.

\begin{figure}
  \centering
  \input{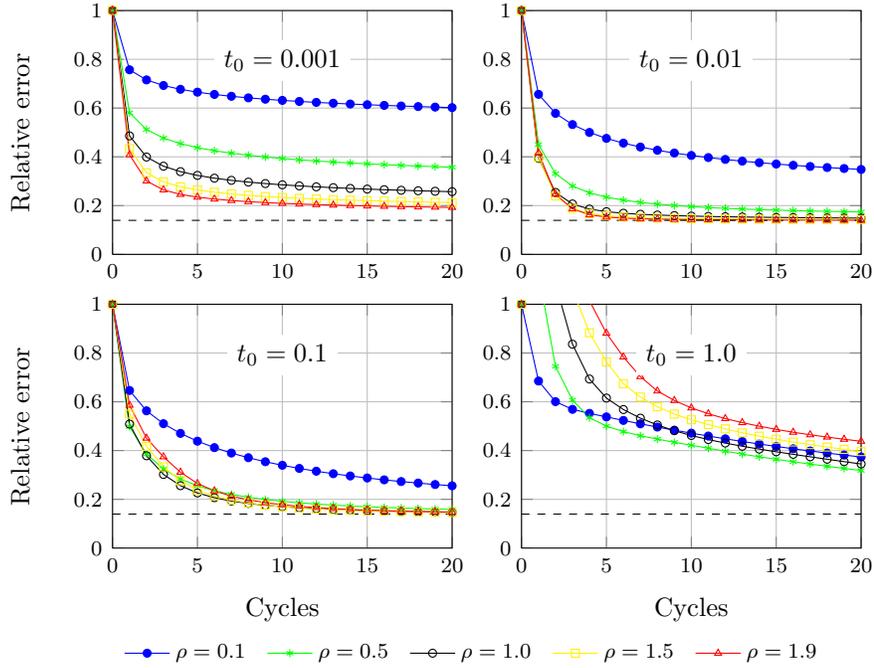}
  \caption{The relative error $\|x_{\ell m} -
    x^{\mathrm{exact}}\|_2/\|x^{\mathrm{exact}}\|_2$ for damped block
    ART with TV regularization after $\ell = 0,1,2,\ldots 20$ full cycles
    with different values
    of $\rho$ and $t_0$. The dashed line marks the norm of the
    relative error of the reference-solution to the TV-regularized least-squares
    problem obtained using the primal--dual first-order method of
    Chambolle and Pock \cite{ChP:11}.}
  \label{fig:ex_bipgtv}
\end{figure}

The plots in Figure~\ref{fig:ex_bipgtv} show our results.
We see that the initial rate of convergence
strongly depends on the choice of both the relaxation parameter $\rho$
and the initial parameter $t_0$.
Nevertheless, with suitable $\rho$ and $t_0$, the method exhibits very fast
initial convergence and achieves a reasonably accurate approximate
solution after only about 10 cycles. In this example, we obtained the
best results with overrelaxation.
This can be seen from the plot in
Figure~\ref{fig:ex_bipgtv_step} which shows the norm of the relative
error after 20 cycles for different values of $\rho$ and $t_0$. It is
also clear from this plot that, in this example, the overrelaxed method finds a
reasonably accurate approximate solution within 20 cycles when $t_0$
is between approximately 0.01 and 0.1.

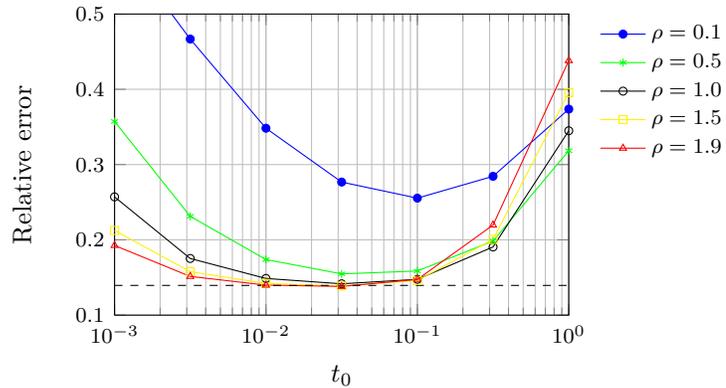
\begin{figure}
  \centering
  \begin{tikzpicture}
 
  \begin{semilogxaxis}[
    name=plot1,
    font=\scriptsize,
    xlabel={\small $t_0$},
    ylabel={\small Relative error},
    xmin=1e-3,
    xmax=1e0,
    ymin=0.1,
    ymax=0.5,
    width=3.0in,
    height=2.2in,
    grid=both,
    legend style={overlay, draw=none,cells={anchor=west},anchor=north west, at={(1.05,1)},font=\scriptsize}
    ]

    \addplot [color=blue,mark=*,mark size=1.5pt] coordinates {
(1.0000e-03,6.0189e-01)(3.1623e-03,4.6671e-01)(1.0000e-02,3.4828e-01)(3.1623e-02,2.7676e-01)(1.0000e-01,2.5543e-01)(3.1623e-01,2.8443e-01)(1.0000e+00,3.7367e-01)
    };
    \addlegendentry{$\rho=0.1$}
    \addplot [color=green,mark=asterisk,mark size=1.5pt] coordinates {
(1.0000e-03,3.5729e-01)(3.1623e-03,2.3126e-01)(1.0000e-02,1.7383e-01)(3.1623e-02,1.5477e-01)(1.0000e-01,1.5862e-01)(3.1623e-01,1.9843e-01)(1.0000e+00,3.1840e-01)
    };
    \addlegendentry{$\rho=0.5$}
    \addplot [color=black,mark=o,mark size=1.5pt] coordinates { 
(1.0000e-03,2.5708e-01)(3.1623e-03,1.7525e-01)(1.0000e-02,1.4872e-01)(3.1623e-02,1.4163e-01)(1.0000e-01,1.4753e-01)(3.1623e-01,1.9058e-01)(1.0000e+00,3.4493e-01)
    };
    \addlegendentry{$\rho=1.0$}
    \addplot [color=yellow,mark=square,mark size=1.5pt] coordinates {
(1.0000e-03,2.1250e-01)(3.1623e-03,1.5770e-01)(1.0000e-02,1.4214e-01)(3.1623e-02,1.3871e-01)(1.0000e-01,1.4596e-01)(3.1623e-01,2.0053e-01)(1.0000e+00,3.9562e-01)
    };
    \addlegendentry{$\rho=1.5$}
    \addplot [color=red,mark=triangle,mark size=1.5pt] coordinates { 
(1.0000e-03,1.9258e-01)(3.1623e-03,1.5142e-01)(1.0000e-02,1.3988e-01)(3.1623e-02,1.3810e-01)(1.0000e-01,1.4745e-01)(3.1623e-01,2.1966e-01)(1.0000e+00,4.3812e-01)
    };
    \addlegendentry{$\rho=1.9$}

   \addplot [color=black, dashed] coordinates {(1e-3,1.3949e-01)(1e0,1.3949e-01)};
  \end{semilogxaxis}

\end{tikzpicture}
  \caption{The relative error $\|x_{\ell m} -
    x^{\mathrm{exact}}\|_2/\|x^{\mathrm{exact}}\|_2$ for damped block
    ART with TV regularization after $\ell = 20$ cycles for different
    values of $t_0$. The dashed line marks the relative error of the reference.}
  \label{fig:ex_bipgtv_step}
\end{figure}

\section{Conclusions}\label{sec-conclusions}
This work contributes to existing knowledge on row-action methods by
providing an extension of the incremental proximal gradient framework
of Bertsekas. By adding a relaxation parameter this framework, we have
shown that it is possible to interpret many well-known row-action
methods for tomographic imaging as incremental methods for solving
some convex optimization problem. More importantly, the framework also
allows us to derive new \emph{generalized} row-action methods that are
based on \emph{generalized projections} (\ie, proximal
operators). We demonstrated this with several examples, including
new ART-like methods for robust regression and regularized regression.

Our numerical experiments suggest that with suitably chosen
parameters, the relaxed incremental proximal gradient methods can
obtain good approximate solutions in a small number of cycles, even
for problems that involve a nonsmooth regularization term such as
TV\@. Interestingly, in most cases we obtained the best results using
either under- or overrelaxation which underlines the practical
importance of relaxation. However, the question of how to choose
algorithm parameters for a given problem remains a difficult one, and
further work is needed to investigate this issue.



\appendix
\section{Proof of Proposition~\ref{prop-ripg-bound}}
\label{app-proof-prop}

We start by proving \eqref{e-prop-ripg-bound} for the
iteration \eqref{e-ripg1}. Let $y$ denote a vector that belongs to
$\mathcal C$. Then, using the nonexpansiveness of the projection
operator $\PC$ and \eqref{e-ripg1-3}, we have
\begin{align} \label{e-ripg1-proof-1}
  \| x_{k+1} - y \|_2^2 &\leq \rho^2 \|z_{k} - y\|_2^2 + (1-\rho)^2
  \|x_{k} - y\|_2^2 + 2\rho(1-\rho)(z_{k} - y)^T(x_{k} - y).
\end{align}
It follows from \eqref{e-ripg1-1} that $w_k = x_k - t_k \tnabla
g_{i_k}(w_k)$ for some $\tnabla g_{i_k}(w_k) \in \partial
g_{i_k}(w_k)$, and combining this with \eqref{e-ripg1-2}, we get
\begin{align*}
  z_k - y = x_k - y - t_k \tnabla g_{i_k}(w_k) - t_k \tnabla h_{i_k}(w_k).
\end{align*}
Taking inner products with first $z_k-y$ and then $x_k - y$ on both sides of this
equation, we obtain the following expression
\begin{align*}
  2(z_k-y)^T(x_k-y) = \|x_k-y\|_2^2 + \|z_k-y\|_2^2 +
  t_k\tnabla f_{i_k}(w_k)^T (z_k-x_k)
\end{align*}
where $\tnabla f_{i_k}(w_k) = \tnabla g_{i_k}(w_k) +
\tnabla h_{i_k}(w_k)$.
 Using this result in
\eqref{e-ripg1-proof-1} and by substituting $x_k -t_k \tnabla f_{i_k}(w_k)$ for $z_k$, we obtain the inequality
\begin{align}\label{e-ripg1-proof-2}
\| x_{k+1} - y\|_2^2 &\leq \rho \| x_k  - y -t_k \tnabla f_{i_k}(w_k) \|_2^2 +
(1-\rho) \|x_{k} - y \|_2^2 - \rho
(1-\rho) t_k^2 \| \tnabla f_{i_k}(w_{k}) \|_2^2.
\end{align}
Expanding the first term on the right-hand side of this inequality yields
\begin{align}\label{e-ripg1-proof-3}
\rho \| x_k -y  - t_k \tnabla f_{i_k}(w_k)\|_2^2 = \rho \|x_k -y
  \|_2^2 + \rho \|\tnabla f_{i_k}(w_k) \|_2^2 - 2\rho t_k \tnabla f_{i_k}(w_k)^T(x_k - y)
\end{align}
and furthermore, using the definition of a subgradient of $f_{i_k}$ at $w_k$, the
last term on the right-hand side of \eqref{e-ripg1-proof-3} can be
bounded from above as
\begin{align}\label{e-ripg1-proof-4}
  -2\rho t_k \tnabla f_{i_k}(w_k)^T(x_k-y) \leq -2\rho t_k
  (f_{i_k}(w_k) - f_{i_k}(y)) + 2\rho t_k \tnabla
  f_{i_k}(w_k)^T(w_k - x_k) .
\end{align}
Combining \eqref{e-ripg1-proof-3}, \eqref{e-ripg1-proof-4}, and
\eqref{e-ripg1-proof-2}, we get
\begin{align}
\notag
    \|x_{k+1}-y\|_2^2 & \leq \|x_{k} - y\|_2^2  - 2\rho t_k (f_{i_k}(w_k)
    - f_{i_k}(y)) \\
\label{e-ripg1-proof-5}
& \qquad + \rho^2 t_k^2 \| \tnabla f_{i_k}(w_{k})
    \|_2^2  + 2\rho t_k \tnabla f_{i_k}(w_{k})^T (w_{k} - x_k )
\end{align}
and using \eqref{e-ripg1-1} and the definition of
$\tnabla f_{i_k}(w_k)$, we can express the last two terms on the
right-hand side as
\begin{align*}
    \rho t_k^2 \left(\rho \|\tnabla
      h_{i_k}(w_k) \|_2^2 -(2-\rho) \|\tnabla g_{i_k}(w_k)\|_2^2 + 2(1-\rho) \tnabla g_{i_k}(w_k)^T\tnabla
      h_{i_k}(w_k)   \right).
\end{align*}
Using the Cauchy-Schwartz inequality and the inequality $\|\tnabla
h_{i_k}(w_k)\|_2 \leq c$ from Assumption \ref{assumption-ripg1}, we
obtain the bound
\begin{multline*}
\rho^2 t_k^2 \| \tnabla f_{i_k}(w_{k})
    \|_2^2  + 2\rho t_k \tnabla f_{i_k}(w_{k})^T (w_{k} - x_k )
    \leq \\
    \rho t_k^2 \left(\rho c^2 -(2-\rho) \|\tnabla g_{i_k}(w_k)\|_2^2 +
      2c |1-\rho | \| \tnabla g_{i_k}(w_k) \|_2   \right).
\end{multline*}
This is a concave function of $\| \tnabla g_{i_k}(w_k) \|_2$ for
$\rho \in [\delta,2-\delta]$, and we obtain a simpler bound
by maximizing over $\|\tnabla g_{i_k}(w_k)\|_2 \leq c$. Thus, if we let
\[ \|  \tnabla g_{i_k}(w_k) \|_2 =  c\, \min\{ |1-\rho|/(2-\rho), 1 \} \]
we obtain the bound
\begin{align}\label{e-ripg1-proof-6}
  \rho^2 t_k^2 \| \tnabla f_{i_k}(w_{k})
    \|_2^2  + 2\rho t_k \tnabla f_{i_k}(w_{k})^T (w_{k} - x_k )
    \leq \rho t_k^2 c^2 \alpha(\rho)
\end{align}
where
\[ \alpha(\rho) =
\begin{cases}
  1/(2-\rho) & \delta \leq \rho \leq 3/2 \\
  4(\rho-1)  & 3/2 \leq \rho \leq 2-\delta.
\end{cases}
\]
Combining \eqref{e-ripg1-proof-5} and \eqref{e-ripg1-proof-6} yields
\begin{align}\label{e-ripg1-proof-7}
\|x_{k+1} - y \|_2^2 & \leq   \|x_{k} - y\|_2^2 - 2\rho t_k(
f_{i_k}(w_{k}) - f_{i_k}(y)) +\rho t_k^2 c^2 \alpha(\rho).
\end{align}
Applying this bound recursively, and since the index sequence $\{i_k\}$
is cyclic, we have
\begin{align}
\notag
  \| x_{k+m} - y\|_2^2 \leq \|x_{k} - y\|_2^2 - 2\rho t_k (f(x_{k}) - f(y)) + m \rho t_k^2 c^2 \alpha(\rho) \\
   +2\rho t_k \sum_{j=1}^m \left(  f_{j}(x_{k}) - f_{j}(w_{k+j-1} )\right).
\end{align}
We can upper bound $f_{j}(x_{k}) - f_{j}(w_{k+j-1} )$ using Assumption \ref{assumption-ripg1}, \ie,
\begin{align}
  f_j(x_{k}) - f_j(w_{k+j-1}) \leq 2c\, \|x_{k} - w_{k+j-1} \|_2.
\end{align}
Furthermore, from the triangle inequality we have
\begin{align}
  \| x_{k} - w_{k+j-1} \|_2 \leq \| x_{k} -x_{k+1} \|_2 + \ldots +
  \|x_{k+j-2} - x_{k+j-1} \|_2 + \|x_{k+j-1} - w_{k+j-1} \|_2
\end{align}
and using \eqref{e-ripg1-proof-1}, the first $j-1$ right-hand side terms can be
bounded using the inequality
\begin{align}
  \| x_{k} - x_{k+1} \|_2 &\leq \rho \| x_{k} - z_{k} \|_2 = \rho t_k \| \tnabla g_{i_k}(w_{k}) + \tnabla
  h_{i_k}(w_{k}) \|_2 \leq 2\rho t_k c.
\end{align}
Similarly, from \eqref{e-ripg1-1}, we have $\| w_{k} - x_{k} \|_2 \leq
t_k c$, and hence
\begin{align}\label{e-ripg1-proof-8}
 \| x_{k} - w_{k+j-1} \|_2 \leq  2(j-1)\rho t_k c + t_k c
\end{align}
and
\begin{align}\label{e-ripg1-proof-9}
  2\rho t_k \sum_{j=1}^m \left(
    f_{j}(x_{k}) - f_{j}(w_{k+j-1} )\right) \leq 4\rho t_k^2 c^2
  \sum_{j=1}^m (2(j-1)\rho + 1) = 4\rho t_k^2 c^2 (\rho m^2 +
  (1-\rho)m).
\end{align}
Combining \eqref{e-ripg1-proof-7}, \eqref{e-ripg1-proof-8}, and
\eqref{e-ripg1-proof-9}, we obtain the desired result \eqref{e-prop-ripg-bound}.

We now prove \eqref{e-prop-ripg-bound} for the iteration \eqref{e-ripg2}.
Using \eqref{e-ripg2-1} and \eqref{e-ripg2-2}, we have that
\begin{align*}
  z_{k} - y = x_{k} - y - t_k \tnabla g_{i_k}(z_{k}) - t_k
\tnabla h_{i_k}(x_{k}).
\end{align*}
Taking inner products on both sides of this equation with first $z_{k}-y$ and then
$x_{k}-y$, we obtain the equations
\begin{align*}
  (z_{k} - y)^T(x_{k} - y) &= \| z_{k} - y\|_2^2 + t_k (z_{k} - y)^T ( \tnabla g_{i_k}(z_{k}) +\tnabla h_{i_k}(x_{k})) \\
(z_{k} - y)^T(x_{k} - y) & = \| x_{k} - y\|_2^2 - t_k(x_{k} - y)^T( \tnabla g_{i_k}(z_{k}) + \tnabla h_{i_k}(x_{k}))
\end{align*}
and adding these yields
\begin{align}\label{e-ripg2-proof-1}
  2(z_{k} - y)^T(x_{k} - y) = \| x_{k} - y\|_2^2 + \| z_{k} -  y\|_2^2 + t_k (z_{k} - x_{k})^T ( \tnabla g_{i_k}(z_{k}) + \tnabla h_{i_k}(x_{k})).
\end{align}
Using this in \eqref{e-ripg1-proof-1} (which holds for both R-IPG1 and
R-IPG2), we have that
\begin{align}
\label{e-ripg2-proof2}
  \| x_{k+1} - y \|_2^2 &\leq \rho \|z_{k} - y\|_2^2 + (1-\rho)
  \|x_{k+1} - y\|_2^2 - \rho(1-\rho) t_k^2\|\tnabla g_{i_k}(z_{k}) + \tnabla h_{i_k}(x_{k})\|_2^2
\end{align}
where the first term on the right-hand side can be expressed as
\begin{align*}
  \rho \left( \|x_k - y\|_2^2 + t_k^2 \|\tnabla g_{i_k}(z_k) + \tnabla
  h_{i_k}(x_k)\|_2^2 - 2(\tnabla g_{i_k}(z_k) + \tnabla
  h_{i_k}(x_k))^T(x_k-y) \right)
\end{align*}
since $z_k = x_k - t_k \tnabla h_{i_k}(x_k) - t_k \tnabla
g_{i_k}(z_k)$. Thus,
\begin{align} \notag
  \| x_{k+1} - y \|_2^2 &\leq  \|x_{k} - y\|_2^2 - 2\rho t_k \tnabla  h_{i_k}(x_{k})^T(x_{k} - y) - 2\rho t_k
\tnabla g_{i_k}(z_{k})^T(z_{k} - y) \\
\notag
& \qquad
+ \rho  t_k^2 \left( \| \tnabla
  h_{i_k}(x_{k}) \|_2^2 -  \|\tnabla g_{i_k}(z_{k}) \|_2^2  - (1-\rho) \|\tnabla g_{i_k}(z_{k}) + \tnabla
h_{i_k}(x_{k})\|_2^2 \right)
\end{align}
and using the definition of the subdifferentials $\partial
g_{i_k}(z_k)$ and $\partial h_{i_k}(x_k)$ together with \eqref{e-ripg1-proof-6}, we obtain
\begin{align}
\notag
 \| x_{k+1} - y \|_2^2 & \leq \|x_{k} - y\|_2^2 -2\rho t_k (f_{i_k}(x_{k}) -
f_{i_k}(y)) + \rho t_k^2 c^2 \alpha(\rho)  \\
& \qquad + 2\rho t_k (g_{i_k}(x_{k}) - g_{i_k}(z_{k})).
\end{align}
Consequently, after a complete cycle (\ie, $m$ iterations), we have
\begin{align}
\notag
  \| x_{k+m} - y \|_2^2& \leq \|x_{k} - y \|_2^2 - 2\rho t_k
  (f(x_{k}) - f(y)) + m\rho t_k^2c^2\alpha(\rho) \\
\notag
  & \qquad +2\rho t_k
  \sum_{j=1}^m \left( f_j(x_{k}) - f_j(x_{k+j-1}) \right)\\
\label{e-ripg2-proof-2}
 &\qquad +2\rho t_k
  \sum_{j=1}^m \left( g_j(x_{k+j-1}) - g_j(z_{k+j-1}) \right).
\end{align}
Now, using Assumption \ref{assumption-ripg2}, we can bound $ f_j(x_{k}) -
f_j(x_{k+j-1})$ above as
\begin{align*}
  f_j(x_{k}) - f_j(x_{k+j-1}) \leq 2 c\, \| x_{k} - x_{k+j-1} \|_2
\end{align*}
where
\begin{align*}
  \| x_{k} - x_{k+j-1} \|_2 \leq \| x_{k} - x_{k+1} \|_2 + \cdots
  + \| x_{k+j-2} - x_{k+j-1} \|_2
\end{align*}
and
\[ \| x_{k} - x_{k+1} \|_2 \leq \rho \|z_{k} - x_{k}\|_2 = \rho t_k \|
\tnabla g_{i_k}(z_{k}) + \tnabla h_{i_k}(x_{k}) \|_2
\leq 2\rho t_k c  \]
and this implies that
\begin{align}\label{e-ripg2-proof-3}
  2\rho t_k
  \sum_{j=1}^m \left( f_j(x_{k}) - f_j(x_{k+j-1}) \right) \leq
  4\rho^2 t_k^2c ^2(m^2-m).
\end{align}
Similarly, using Assumption \ref{assumption-ripg2}, we have that
\begin{align*}
  g_{i_k}(x_{k+j-1}) - g_{i_k}(z_{k+j-1}) &\leq c\, \| x_{k+j-1}
  - z_{k+j-1} \|_2 \\
  &  \leq t_k c \, \| \tnabla g_{i_k}(z_{k+j-1}) +
  \tnabla h_{i_k}(x_{k+j-1}) \|_2 \\
  & \leq 2 t_k c^2
\end{align*}
and hence
\begin{align}\label{e-ripg2-proof-4}
  2\rho t_k
  \sum_{j=1}^m \left( g_j(x_{k+j-1}) - g_j(z_{k+j-1}) \right) \leq
  4m \rho t_k^2 c^2.
\end{align}
Combining \eqref{e-ripg2-proof-2}, \eqref{e-ripg2-proof-3}, and
\eqref{e-ripg2-proof-4}, we get the desired result \eqref{e-prop-ripg-bound}.

\bibliographystyle{plain}
\bibliography{literature}

\end{document}